\documentclass[12pt]{amsart}
\usepackage[margin=0.5in]{geometry}
\newcommand{\R}{\mathbb{R}}

\newcommand{\N}{\mathbb{N}}

\makeatletter
\@namedef{subjclassname@2020}{\textup{2020} Mathematics Subject Classification}
\makeatother

\renewcommand{\epsilon}{\varepsilon}

\newcommand{\e}{\varepsilon}

\usepackage[english]{babel}
\usepackage[backref=page]{hyperref}

\usepackage[alphabetic]{amsrefs}
\usepackage{amsmath,amssymb,amsfonts,amsthm,enumerate}
\usepackage{url}
\usepackage{graphicx,epstopdf,xcolor}
\usepackage{mathtools}

\numberwithin{equation}{section}

\newtheorem{theorem}{Theorem}[section]
\newtheorem{lemma}[theorem]{Lemma}
\newtheorem{definition}[theorem]{Definition}
\newtheorem{remark}[theorem]{Remark}

\newtheorem{proposition}[theorem]{Proposition}
\newtheorem{corollary}[theorem]{Corollary}

\renewcommand{\leq}{\leqslant}
\renewcommand{\le}{\leqslant}
\renewcommand{\geq}{\geqslant}
\renewcommand{\ge}{\geqslant}
\allowdisplaybreaks

\title{Integral operators defined ``up to a polynomial''}

\author{Serena Dipierro}
\address{University of Western Australia,
Department of Mathematics and Statistics, 35 Stirling Highway,
WA6009 Crawley, Australia}
\email{serena.dipierro@uwa.edu.au}

\author{Aleksandr Dzhugan}
\address{Universit\`a di Bologna, Dipartimento di Matematica,
Piazza di Porta San Donato 5, 40126 Bologna, Italy}
\email{aleksandr.dzhugan2@unibo.it}

\author{Enrico Valdinoci}
\address{University of Western Australia,
Department of Mathematics and Statistics, 35 Stirling Highway,
WA6009 Crawley, Australia}
\email{enrico.valdinoci@uwa.edu.au}

\begin{document}

\begin{abstract}
We introduce a suitable notion of integral operators (comprising the fractional Laplacian as a particular case)
acting on functions with minimal requirements at infinity.
For these functions, the classical definition would lead to divergent expressions,
thus we replace it with an appropriate framework obtained by a cut-off procedure. The notion obtained
in this way quotients out the polynomials which produce the divergent pattern once the cut-off is removed.

We also present results of stability under the appropriate
notion of convergence and compatibility results between polynomials of different orders.
Additionally, we address the solvability of the Dirichlet problem.

The theory is developed in general in the pointwise sense. A viscosity counterpart
is also presented under the additional assumption that the interaction kernel has a sign,
in conformity with the maximum principle structure.
\end{abstract}


\maketitle

\section{Introduction}

A classical line of investigation in mathematical analysis and mathematical physics consists in the study
of integro-differential operators. The motivations for this stream of research come both from
theoretical mathematics (such as harmonic analysis, singular integral theory, fractional calculus, etc.)
and concrete problems in applied sciences 
(with questions related to water waves, crystal dislocations, finance, optimization, minimal surfaces, etc.):
see e.g. the introduction in~\cite{BOOKCA} and the references therein for a number of explicit motivations
and examples.

A special focus of this stream of research deals with linear
integro-differential operators of the form
\begin{equation} \label{Au}
Au(x)={\mbox{P.V.}}\int_{\mathbb{R}^n}(u(x)-u(y))K(x,y)\,dy=
\lim_{\e\searrow0}\int_{\mathbb{R}^n\setminus B_\e(x)}(u(x)-u(y))K(x,y)\,dy.
\end{equation}
As customary, the notation~$B_\e(x)$ denotes the open ball of radius~$\e$ centered at the point~$x$
(when~$x$ is the origin, one simply uses the notation~$B_\e$ for short).
The notation ``P.V.'' above (which will be omitted in the rest of this paper for the sake of simplicity)
means ``in the principal value sense'' and takes into account possible integral cancellations.
The action of such operator  is to ``weight'' the oscillations of the function~$u$
according to the kernel~$K$. To make sense of the expression above, two types of assumptions
need to be accounted for:
\begin{itemize}
\item[\checkmark] 
if the kernel~$K$ is singular when~$x=y$,
the function $u$ needs to be regular enough near the point $x$
(to allow integral cancellations and take advantage of the principal value in~\eqref{Au}),
\item[\checkmark] the function $u$ needs to be sufficiently well-behaved at infinity
(namely, its growth has to be balanced by the kernel~$K$
to obtain in~\eqref{Au} a convergent integral at infinity).
\end{itemize}
Roughly speaking, these two conditions correspond to the request that the integral in~\eqref{Au}
converges both in the vicinity of the given point~$x$ and at infinity.
With respect to this,
the regularity condition is necessary for the local convergence of the integral 
and it is common to differential (rather than integral) operators:
in a sense, for differential problems
the regularity of~$u$ ensures
that incremental quotients converge to derivatives
and, somewhat similarly, for integral problems
the regularity of~$u$ allows the increment inside the integral
to compensate the possible singularity of the kernel.
Instead, the second assumption on the behavior of~$u$ at infinity is needed only to guarantee the tail convergence,
it is a merely nonlocal feature
and has no counterpart for the case
of differential operators.

Conditions ``at infinity'' are also technically more difficult to deal with.
First of all, they are more expensive to be computed, since they need to account for
virtually all the values of the given function (while regularity ones
deal with the values in an arbitrarily small region). Furthermore,
these conditions are typically lost after one analyzes the problem at a small scale
(since blow-up procedures alter the behavior of the solutions at infinity,
with the aim of detecting the local patterns).
Moreover, it is sometimes difficult to detect optimal assumptions
for nonlocal problems even in very basic and fundamental questions
(see e.g. the open problem after Theorem~3.2 in~\cite{MR3916700}),
hence any theory based only on ``essential'' assumptions is doomed to have promising
future developments.

It would be therefore very desirable to develop a theory of integral
operators that does not heavily rely on the conditions at infinity
(in spite of the striking fact that these conditions are needed
even in the definition of the operator itself!).
To this end, a theory of ``fractional Laplacian operators up to polynomials''
has been developed in~\cite{MR3988080, DSV19}
to address the case of functions with polynomial growth
(see also~\cite{MR1421222} for related approaches;
see e.g.~\cite{MR2707618, MR2944369, MR3613319, MR3967804, MR4303657}
and the references therein for the basics
on the fractional Laplace operators).
The gist of this method
is to consider the family of cut-offs 
\begin{equation} \label{chi}
\chi_R(x):=\begin{cases}
1,&{\mbox{ if }}x \in B_R;\\
0,&\mbox{otherwise}
\end{cases}
\end{equation}
and apply the operator to the function~$\chi_R u$. Of course, in general,
it is not possible to send~$R\to+\infty$, since the operator is not well-defined on~$u$,
nevertheless it is still possible to perform such an operation once an appropriate polynomial
is ``taken out'' from the equation. Given the ``rigidity'' of the space of polynomials
(which is finite dimensional and easily computable) the method is flexible
and solid, it produces interesting results and can be efficiently combined with blow-up procedures,
see~\cite{MR3770173, MR4038144}.\medskip

The goal of this note is twofold: on the one hand, we review and extend
the theory developed in~\cite{MR3988080, DSV19}, on the other hand, we generalize
the previous setting in order to include much more general classes of kernels
(in particular, kernels which are not necessarily scaling invariant).
Besides its interest in pure mathematics, this generalization
has a concrete impact on the study of
interaction potentials of interatomic type arising in molecular mechanics and materials science, such as
the Morse potential~\cite{MORS}
\begin{equation}\label{E-dX-PO-2} K(x,y)= e^{-2(|x-y|-1)}-e^{-(|x-y|-1)},\end{equation}
the Buckingham potential~\cite{BUCK}
\begin{equation}\label{E-dX-PO-3} K(x,y)=e^{-|x-y|}-\frac1{|x-y|^6},\end{equation}
as well as their desingularized forms obtained by setting~$K_\e(x,y):=\min\left\{\frac1\e,\,K(x,y)
\right\}$.
Other classical potentials arising in probability and modelization include also the Gauss kernel
\begin{equation}\label{E-dX-PO-4}
K(x,y)=e^{-|x-y|^2},\end{equation}
the Abel kernel
\begin{equation}\label{E-dX-PO-5}
K(x,y)=e^{-|x-y|},\end{equation}
the mollification kernel
\begin{equation}\label{E-dX-PO-6}
K(x,y)=\begin{cases}e^{-\frac1{1-|x-y|^2}} & {\mbox{ if }}|x-y|<1,
\\0&{\mbox{ if }}|x-y|\geq 1\end{cases}
\end{equation}
and the class of kernels comparable to that of the fractional
Laplacian
\begin{equation}\label{FL}
\frac{\lambda}{|x-y|^{n+2s}}\le K(x,y)\le \frac{\Lambda}{|x-y|^{n+2s}}
\end{equation}
for~$s\in(0,1)$ and~$\Lambda\ge \lambda >0$.
\medskip

The theory of integral operators that we develop
is broad enough to include the kernels above (and others as well) into a unified setting.
The operators will be suitably defined
``up to a polynomial'', in a sense that will be made precise in Definition~\ref{def}.
This framework relies on a suitable decomposition of the integral operator
with respect to cut-off functions that is showcased in Theorem~\ref{cut-off}.
This setting is stable under the appropriate
notion of convergence, as it will be detailed in Proposition~\ref{forthprop},
and it presents nice compatibility results between polynomials of different orders,
as it will be pointed out in
Corollary~\ref{corequiv} and Lemma~\ref{lemmaj}.
We also stress that the generalized notion of operators that we deal with
is ``as good as the classical one'' in terms of producing solutions
for the associated Dirichlet problem: indeed,
as it will be clarified in 
Theorem~\ref{KS:098iKS-904596}, the solvability
of the classical Dirichlet problem in the class of functions
with nice behavior at infinity is sufficient to ensure the solvability
of the generalized Dirichlet problem for the operator defined ``up to a polynomial''.\medskip

To develop this theory, we mainly focused on the case of sufficiently smooth
(though not necessarily well-behaved at infinity) functions. This choice was dictated by
three main reasons. First of all, we aimed at developing the core of the theory by focusing
on its essential features, rather than complicating it by additional difficulties.
Moreover, we intended to split the complications arising
from the possible lack of smoothness of the solutions with those produced by their behavior at infinity,
consistently with the initial discussion
presented right after~\eqref{Au}. Additionally, we stress that
the generality of kernels addressed by our theory
goes far beyond the ones of ``elliptic'' type, therefore a comprehensive
regularity theory does not hold in such an extensive framework.\medskip

However, one can also recast our theory in terms of viscosity solutions.
For this, since viscosity theory relates to maximum principles,
one needs the additional assumption that the kernel has a sign.
In particular, in this context
one can obtain a viscosity definition of operators ``up to a polynomial''
(see Definition~\ref{defv}) and discuss its stability properties under uniform convergence
(see Lemma~\ref{lemCSv}) and the consistency properties with respect of polynomials of different degree
(see Corollary~\ref{corequivv}
and Lemma~\ref{lemmajv}).
When the structure is compatible with both settings,
the pointwise framework and the viscosity one are essentially equivalent
(see Lemma~\ref{EQUIBVA}).
Furthermore, for kernels comparable with that of the fractional Laplacian
a complete solvability of the Dirichlet problem can be obtained
(see Theorem~\ref{EQUI90-97900BVA}).\medskip

In the forthcoming Section~\ref{SEC-1} we introduce the main
definitions for integral operators ``up to a polynomial''
and present their fundamental properties. The corresponding Dirichlet problem
will be discussed in Section~\ref{SEC-2}. 

While Sections~\ref{SEC-1} and~\ref{SEC-2}
focus on the pointwise definition of this generalized notion of operators,
we devote Section~\ref{VIS-PJN-c} to the corresponding viscosity theory.

\section{Definitions and main properties of operators ``up to a polynomial''}\label{SEC-1}

The mathematical setting in which we work is the following.
For every~$\vartheta\in[0,2]$,
we define~${\mathcal{C}}_\vartheta$ as the set of
functions~$u\in L^1_{{\rm{loc}}}(\R^n)$ such that
\begin{equation}\label{DEF:SPAZ} u\in
\begin{cases}C(B_4)\cap
L^\infty(B_4) &{\mbox{ if }} \vartheta=0,\\
C^\vartheta(B_4) &{\mbox{ if }} \vartheta\in(0,1),\\
C^{0,1}(B_4) &{\mbox{ if }} \vartheta=1,\\
C^{1,\vartheta-1}(B_4) &{\mbox{ if }} \vartheta\in(1,2),\\
C^{1,1}(B_4) &{\mbox{ if }} \vartheta=2.
\end{cases}\end{equation}
Furthermore, for all~$m\in\N_0$ and all~$\vartheta\in[0,2]$, we
introduce~${\mathcal{K}}_{m,\vartheta}$ as the space of kernels~$K=K(x,y)$
such that\footnote{As customary,
in this paper we denote by~$\Omega^c$
the complementary set $\mathbb{R}^n \setminus \Omega$ for a given $\Omega \subseteq \mathbb{R}^n$.}
\begin{eqnarray}
\label{taylor}
&&{\mbox{for all~$y\in B_3^c$, the map~$x\in B_1\mapsto K(x,y)$ is~$C^m(B_1)$}}
\\
\label{locint}{\mbox{and }}
&&\int_{\mathbb{R}^n}\min\{|x-y|^\vartheta,1\} \,|K(x,y)|\,dy<+\infty\quad
{\mbox{ for all $x \in B_1$.}}
\end{eqnarray}\
If~$\vartheta\in(1,2]$ we require additionally that every~$K\in{\mathcal{K}}_{m,\vartheta}$
satisfies\footnote{We observe that these assumptions are satisfied
by the kernel in~\eqref{E-dX-PO-2} for every~$n$,
by the kernel in~\eqref{E-dX-PO-3} with~$n=5$ and by all the corresponding desingularized kernels for every~$n$.
The kernels in~\eqref{E-dX-PO-4}, \eqref{E-dX-PO-5} and~\eqref{E-dX-PO-6}
also satisfy these assumptions for every~$n$. The kernel in~\eqref{FL} satisfies~\eqref{locint} for every~$n$
and every~$\vartheta\in(2s,2]$.}
\begin{equation}\label{assym}
K(x,x+z)=K(x,x-z)\quad {\mbox{ for all~$x$, $z \in B_1$}}.\end{equation}
Given~$K\in{\mathcal{K}}_{m,\vartheta}$, we consider the space~${\mathcal{C}}_{\vartheta,K}$
of all the functions~$u\in{\mathcal{C}}_\vartheta$ for which
\begin{eqnarray}
\label{ring}
&&\sum_{|\alpha|\le m-1}\int_{B_R\setminus B_3}
|u(y)|\,|\partial^\alpha_x K(x,y)|\,dy<+\infty
\quad {\mbox{ for all~$R>3$ and $x \in B_1$}}\\
\label{mcond}{\mbox{and }}
&&\int_{B^c_3} |u(y)|\sup_{{|\alpha|=m}\atop{x\in B_1}}
|\partial^\alpha_x K(x,y)|\,dy<+\infty.
\end{eqnarray}

In this setting, we have the following results that show the role played by a cut-off
function in the computation of the operator in~\eqref{Au} on functions in~${\mathcal{C}}_{\vartheta,K}$:

\begin{theorem} \label{cut-off}
Let\footnote{In this paper, we use the notation~$\mathbb{N}_0$ to denote the set of natural numbers including zero, that is $\{0,1,2,3,\dots\}$. The notation~$\mathbb{N}$ is instead reserved for the strictly positive natural numbers~$\{1,2,3,\dots\}$,
therefore~$\mathbb{N}_0=
\mathbb{N} \cup 0$.
For~$m=0$, the result in Theorem~\ref{cut-off} holds true,
simply by taking~$P_{u,\tau}:=0$ and~$\psi(x,y):=K(x,y)$.} $m \in \mathbb{N}_0$, $\vartheta\in[0,2]$,
$K\in {\mathcal{K}}_{m,\vartheta}$
and~$u\in {\mathcal{C}}_{\vartheta,K}$.
Let also~$\tau:\mathbb{R}^n \to [0,1]$ be
compactly supported and such that~$\tau = 1$ in $B_3$.

Then, there exist a
function $f_{u,\tau}:\mathbb{R}^n\to \mathbb{R}$
and a polynomial $P_{u,\tau}$ of degree at most $m-1$ such that 
\begin{equation} \label{p+f}
A(\tau u)=P_{u,\tau}+f_{u,\tau}
\end{equation}
in $B_1$.
In addition, $f_{u,\tau}$ can be written in the following form: there exists $\psi: B_1 \times B_3^c \to \mathbb{R}$, with\footnote{Notice that
the right hand side of~\eqref{psibound} may be infinite and in this case~\eqref{psibound} is obviously true.
On the contrary, if the right hand side of~\eqref{psibound} is finite, then the quantity on the left hand side
of~\eqref{psibound} is bounded as well.}
\begin{equation} \label{psibound}
\sup_{x \in B_1}|\partial^\gamma_x\psi(x,y)|\le  C\,
\sup_{{x\in B_1}\atop{m\le|\eta|\le m+|\gamma|}}|\partial^\eta_x K(x,y)|,
\end{equation}
for every~$\gamma\in{\mathbb{N}}^n$ and for a suitable constant~$C>0$ depending
on~$m$, $n$ and~$|\gamma|$,
such that
\begin{equation} \label{repr}
f_{u,\tau}=f_{1,u}+f_{2,u}+f^*_{u,\tau},
\end{equation}
where
\begin{equation} \label{f12*}
\begin{split}
&f_{1,u}(x):=\int_{B_3} \big(u(x)-u(y)\big)K(x,y)\,dy  ,  \\
&f_{2,u}(x):=u(x)\int_{B^c_3}K(x,y)\,dy     \\
{\mbox{and }}\qquad &f^*_{u,\tau}(x):=\int_{B^c_3}\tau(y) u(y)\psi(x,y)\,dy.
\end{split}
\end{equation}
\end{theorem}

\begin{proof}
We observe that
\begin{equation}\label{PAC1}
{\mbox{$f_{1,u}$ and~$f_{2,u}$ are well-defined and finite
for every~$x\in B_1$. }}\end{equation}
To check this, we first consider the case in which~$\vartheta\in[0,1]$. 
In this case, for every~$x\in B_1$ and~$y\in B_3$,
$$|u(x)-u(y)|\le C |x-y|^\vartheta,$$
for some~$C>0$, and thus
\begin{eqnarray*}
&& |f_{1,u}(x)|\le
\int_{B_3}\big|u(x)-u(y)\big|\, |K(x,y)|\,dy\le C\int_{B_3}|x-y|^\vartheta\, |K(x,y)|\,dy\\&&\qquad\le C\left(
\int_{B_3\cap B_1(x)}|x-y|^\vartheta\,|K(x,y)|\,dy+
\int_{B_3\setminus B_1(x)}|K(x,y)|\,dy\right),
\end{eqnarray*}
up to renaming~$C>0$, and this shows that~$f_{1,u}$ is
well-defined and finite, thanks to~\eqref{locint}.

If instead~$\vartheta\in(1,2]$ we claim that, since~$u \in {\mathcal{C}}_\vartheta$,
there exists a constant $L>0$ such that for all~$|z|$ sufficiently
small (say~$z\in B_1$) we have that
\begin{equation}\label{sayzinbi}
|2u(x)-u(x+z)-u(x-z)|\le L|z|^\vartheta.\end{equation}
Indeed, in this case we know that~$u\in C^{1,\vartheta-1}(B_4)$
and thus
\begin{equation}\begin{split}\label{siw74b8437654693}
&|2u(x)-u(x+z)-u(x-z)|= |(u(x)-u(x+z))+(u(x)-u(x-z))|
\\&\qquad=\left|-\int_{0}^1 \nabla u(x+tz)\cdot z\,dt+
\int_{0}^1 \nabla u(x-tz)\cdot z\,dt
\right|
\\
&\qquad\le \int_{0}^1 |\nabla u(x+tz)-\nabla u(x-tz)|\,|z|\,dt\\
&\qquad\le C \int_{0}^1 |x+tz-(x-tz)|^{\vartheta-1}|z|\,dt=
C \int_{0}^1 t^{\vartheta-1}|z|^{\vartheta-1}|z|\,dt =C|z|^\vartheta,
\end{split}\end{equation}
up to relabeling~$C$ at every step. This establishes~\eqref{sayzinbi}.

Now, we notice that
\begin{equation} \label{loc}
\begin{split}& f_{1,u}(x)=
\int_{B_3}\big(u(x)-u(y)\big)K(x,y)\,dy \\
&\qquad=\int_{B_1(x)}\big(u(x)-u(y)\big)K(x,y)\,dy
+\int_{B_3\setminus B_1(x)}\big( u(x)- u(y)\big)K(x,y)\,dy\\&\qquad
=:I_1+I_2.
\end{split}
\end{equation}
Using \eqref{assym} and~\eqref{sayzinbi}, we obtain that
\begin{eqnarray*}&&
\left|\int_{B_1(x)}\big(u(x)-u(y)\big)K(x,y)\,dy\right|
\\&&\quad=
\frac{1}{2}\left|\int_{B_1}
\big(u(x)-u(x+z)\big)K(x,x+z)\,dz+
\int_{B_1} \big(u(x)-u(x-z)\big)K(x,x-z)\,dz\right|
\\&&\quad=
\frac{1}{2}\left|\int_{B_1}
\big(2u(x)-u(x+z)-u(x-z)\big)K(x,x+z)\,dz\right|
\\&&\quad\le \frac12\int_{B_1}
\big|2u(x)-u(x+z)-u(x-z)\big|\, |K(x,x+z)|\,dz\\&&\quad \le
\frac{L}2\int_{B_1}|z|^\vartheta\, |K(x,x+z)|\,dz\\&&\quad
\le \frac{L}2\int_{B_1}|z|^\vartheta\,|K(x,x+z)|\,dz.
\end{eqnarray*}
As a consequence,
\begin{equation}\label{ehdfweurui056485}
|I_1|
\le \frac{L}2\int_{B_1}|z|^\vartheta\,|K(x,x+z)|\,dz,
\end{equation}
which is finite, thanks to \eqref{locint}.

Furthermore, 
\begin{eqnarray*}
|I_2|\le \int_{B_3\setminus B_1(x)}\big( |u(x)|+| u(y)|\big)\,|K(x,y)|\,dy\le
2\|u\|_{L^\infty(B_4)} \int_{B_3\setminus B_1(x)}|K(x,y)|\,dy,
\end{eqnarray*}
which is
finite, in light of~\eqref{locint}.
This, together with~\eqref{loc}
and~\eqref{ehdfweurui056485}, proves that~$f_{1,u}$
is well-defined and finite in the case~$\vartheta\in(1,2]$.

Also, $f_{2,u}$ is well-defined and finite for every~$\vartheta\in[0,2]$,
thanks to~\eqref{locint}. These observations establish~\eqref{PAC1}.

As a consequence, for any $x \in B_1$, we can write
\begin{equation} \label{atu}
\begin{split}
&A(\tau u)(x)\\=\;&\int_{B_3}\big((\tau u)(x)-(\tau u)(y)\big)K(x,y)\,dy
+\int_{B^c_3}\big((\tau u)(x)-(\tau u)(y)\big)K(x,y)\,dy \\
=\;&\int_{B_3}\big(u(x)-u(y)\big)K(x,y)\,dy
+u(x)\int_{B^c_3}K(x,y)\,dy-\int_{B^c_3}(\tau u)(y)K(x,y)\,dy\\
=\;& f_{1,u}(x)+f_{2,u}(x)-\int_{B^c_3}(\tau u)(y)K(x,y)\,dy.
\end{split}
\end{equation}

Now, in light of the assumption in~\eqref{taylor}, we are allowed to use Proposition~5.34
in~\cite{chierchia} (see also e.g. Theorem~4 on page 461 of~\cite{MR2033094}) and we find that
$$K(x,y)=\sum_{|\alpha|\le m-1}\partial^{\alpha}_xK(0,y)\frac{x^{\alpha}}{\alpha!}-\psi(x,y),$$
where
\begin{equation}\label{psidef}
\psi(x,y):=-\sum_{|\alpha|=m}\frac{m\,x^\alpha}{\alpha!}\int_0^1(1-t)^{m-1}
\partial^\alpha_x K(tx,y )\,dt
.\end{equation}
As a consequence,
\begin{equation} \label{oi3wi5v9b76}
\begin{split}
&\int_{B^c_3}(\tau u)(y)K(x,y)\,dy=\int_{B^c_3} (\tau u)(y)\left(\sum_{|\alpha|\le m-1}\partial^{\alpha}_xK(0,y)\frac{x^{\alpha}}{\alpha!}-\psi(x,y)\right)\,dy \\
&\qquad
=\sum_{|\alpha|\le m-1}\left(\int_{B^c_3} (\tau u)(y)\partial^{\alpha}_x K(0,y)\,dy\right)\frac{x^{\alpha}}{\alpha!}
-\int_{B^c_3}(\tau u)(y)\psi(x,y)\,dy.
\end{split}
\end{equation}
Now, we set, for every $|\alpha|\le m-1$,
\begin{equation} \label{theta}
\theta_{\tau,\alpha}:=\int_{B^c_3} (\tau u)(y)\frac{\partial^{\alpha}_x K(0,y)}{\alpha!}\,dy.
\end{equation}
Suppose that the support of~$\tau$ is contained in some ball~$B_R$
with~$R>3$, and thus
\begin{equation} \label{finitetheta}
 |\theta_{\tau,\alpha}|\le
\int_{B_R\setminus B_3} \left|(\tau u)(y)\frac{\partial^{\alpha}_x K(0,y)}{\alpha!}\right|\,dy\le
\frac{1}{\alpha!}
\int_{B_R\setminus B_3} |u(y)|\, \big|\partial^{\alpha}_x K(0,y)\big|\,dy.
\end{equation}
We stress that the coefficients $\theta_{\tau,\alpha}$ are well-defined, thanks to~\eqref{ring}.

Hence, setting
\begin{equation} \label{P}
P_{u,\tau}(x):=-\sum_{|\alpha|\le m-1} \theta_{\tau, \alpha} x^{\alpha}
\end{equation}
we have that~$P_{u,\tau}$ is a polynomial in~$x$ of degree
at most~$m-1$.
Plugging this information into~\eqref{oi3wi5v9b76}, we obtain that
\begin{equation*}
\int_{B^c_3}(\tau u)(y)K(x,y)\,dy=
-P_{u,\tau}(x)-
\int_{B^c_3}(\tau u)(y)\psi(x,y)\,dy.
\end{equation*}
Now, we notice that, for all $x \in B_1$ and all~$y\in B_3^c$,
$$
|(\tau u)(y)\psi(x,y)|\le
C|u(y)|\sup_{{|\alpha|=m}\atop{z\in B_1}}|\partial^\alpha_x K(z,y )|,
$$
for some~$C>0$, possibly depending on~$m$ and~$n$.
The last function lies in $L^1(B_3^c)$, thanks to~\eqref{mcond}, and
therefore, recalling the definition of~$f^*_{u,\tau}$ in~\eqref{f12*},
we have that
\begin{equation}\label{PAC2}
{\mbox{$f^*_{u,\tau}$ is well-defined and finite.}}\end{equation}
With this setting, we have that
$$ 
\int_{B^c_3}(\tau u)(y)K(x,y)\,dy=
-P_{u,\tau}(x)-f^*_{u,\tau}(x),$$
and therefore, plugging this information into~\eqref{atu}
and recalling~\eqref{repr}, we obtain~\eqref{p+f}.

Hence, to complete the proof of Theorem~\ref{cut-off},
it remains to check~\eqref{psibound}.
For this, recalling the definition of~$\psi$ in~\eqref{psidef},
we have that, for all~$x \in B_1$ and all $y \in B_3^c$,
\begin{eqnarray*}
\partial^\gamma_x \psi(x,y)&=&
\sum_{|\alpha|=m}\int_0^1c_\alpha(t)\,\partial^\gamma_x\left(x^\alpha\,
\partial^\alpha_x K(tx,y )\right)\,dt\\&=&
\sum_{|\alpha|=m}\int_0^1c_\alpha(t)\,\sum_{\beta\le\gamma}\binom{\gamma}{\beta} 
\partial^\beta_x(x^\alpha)\,\partial_x^{\gamma-\beta}
(\partial^\alpha_x K(tx,y ))\,dt\\&=&
\sum_{|\alpha|=m}\int_0^1c_\alpha(t)\,\sum_{\beta\le\gamma} \binom{\gamma}{\beta} 
\partial^\beta_x(x^\alpha)\,t^{|\gamma-\beta|}\,
\partial^{\alpha+\gamma-\beta}_x K(tx,y )\,dt
,\end{eqnarray*}
where~$c_\alpha(t):=\frac{m}{\alpha!}(1-t)^{m-1}$.
Here~$\beta\le\gamma$ means that~$\beta_1\le\gamma_1$, $\cdots$, $\beta_n\le\gamma_n$ and\
we used the notation
$$  \binom{\gamma}{\beta} = \binom{\gamma_1}{\beta_1}\times \cdots  \times\binom{\gamma_n}{\beta_n}.$$ 
Hence,
$$ |\partial^\gamma_x \psi(x,y)|\le C \sup_{{z\in B_1}\atop{m\le|\eta|\le m+|\gamma|}}
|\partial_x^\eta K(z,y)|,$$
for some~$C>0$ depending on~$m$, $n$ and~$|\gamma|$. This establishes~\eqref{psibound}.
\end{proof}

\begin{corollary} \label{corlim}
Let $m \in \mathbb{N}_0$, $\vartheta \in [0,2]$,
$K\in {\mathcal{K}}_{m,\vartheta}$,
$u \in \mathcal{C}_{\vartheta,K}$
and~$R>3$. Let $\tau_R: \mathbb{R}^n \to [0,1]$ be supported in $B_R$,
with~$\tau_R=1$ in~$B_3$, and such that
\begin{equation} \label{limtau}
\lim_{R \to +\infty}\tau_R=1 \mbox{ a.e. in }\mathbb{R}^n.
\end{equation}
Then, there exist a function $f_u: \mathbb{R}^n\to \mathbb{R}$ and a family of polynomials $P_{u,\tau_R}$, which have degree at most $m-1$, such that
\begin{equation} \label{limA}
\lim_{R\to +\infty}[A(\tau_R u)(x)-P_{u,\tau_R}(x)]=f_u(x)
\end{equation}
for any $x\in B_1$. More precisely, we have that
\begin{equation} \label{repr2}
f_u=f_{1,u}+f_{2,u}+f_{3,u},
\end{equation}
where $f_{1,u}$ and $f_{2,u}$ are as in \eqref{f12*} and
\begin{equation} \label{f3}
f_{3,u}(x):=\int_{B^c_3}u(y)\psi(x,y)\,dy.
\end{equation}
\end{corollary}

\begin{proof} 
We apply Theorem \ref{cut-off} with $\tau:= \tau_R$ for any fixed $R$, and then send $R \to +\infty$.
Indeed, by~\eqref{psibound} (used here with~$\gamma:=0$),
for any $x \in B_1$ and $y \in B_3^c$ we have
$$\big|
(\tau_R u)(y)\psi(x,y)\big|\le  C|u(y)|\sup_{{ |\eta|=m}\atop{ z\in B_1 }}|\partial^\eta_x K(z,y)|,$$
for some $C>0$ and the latter function of $y$ lies in $L^1(B_3^c)$, thanks to \eqref{mcond}.

Consequently, we use~\eqref{f12*}, \eqref{limtau} and the
Dominated Convergence Theorem, thus obtaining that
$$\lim_{R \to +\infty}f^*_{u,\tau_R}=\lim_{R \to +\infty}\int_{B_3^c}(\tau_R u)(y)\psi(x,y)\,dy=\int_{B_3^c} u(y)\psi(x,y)\, dy=f_{3,u}(x).$$
Accordingly, taking the limit in \eqref{p+f} we obtain~\eqref{limA}.
Also, the claims in~\eqref{repr2}
and~\eqref{f3} follow\footnote{It is also interesting
to point out that, when~$\tau_R:=\chi_R$,
\label{FPTUNIF}
the limit in~\eqref{limA}
is uniform for~$x\in B_1$. Indeed, by~\eqref{p+f}
and~\eqref{psibound}, for all~$R_2>R_1>4$,
\begin{eqnarray*} &&\sup_{x\in B_1}\Big| [A(\tau_{R_1} u)(x)-P_{u,\tau_{R_1}}(x)]-
[A(\tau_{R_2} u)(x)-P_{u,\tau_{R_2}}(x)]\Big|\\
&\leq&\sup_{x\in B_1}\int_{B_{R_2}\setminus B_{R_1}} | u(y)|\,|\psi(x,y)|\,dy\\
&\leq&
C\int_{B_{R_2}\setminus B_{R_1}}
|u(y)|\sup_{{ |\eta|=m}\atop{ z\in B_1 }}|\partial^\eta_x K(z,y)|
,\end{eqnarray*}
which is as small as we wish, owing to~\eqref{mcond}.
This observation will be further expanded in Lemma~\ref{cor37}.}
from~\eqref{f12*}.
\end{proof}

We are now ready to introduce the formal setting to deal with general operators
defined ``up to a polynomial'':

\begin{definition} \label{def} Let~$m\in\N_0$,
$\vartheta \in [0,2]$, $K\in {\mathcal{K}}_{m,\vartheta}$,
$u \in \mathcal{C}_{\vartheta,K}$ and $f : B_1 \to \R$ be bounded and continuous. 
We say that 
$$
Au\stackrel{m}{=}f \quad {\mbox{ in }}B_1
$$
if there exist a family of polynomials $P_R$, with $\deg P_R \le m-1$, and functions $f_R: B_1 \to \R$ such that
\begin{equation} \label{decomp}
A(\chi_R u)=f_R+P_R
\end{equation}
in $B_1$, with
\begin{equation} \label{limf}
\lim_{R \to +\infty} f_R(x) =f(x).
\end{equation}
\end{definition}

\begin{remark} {\rm
We observe that~\eqref{decomp} is considered here in the pointwise sense.
This is possible, since the setting in~\eqref{DEF:SPAZ}
suffices for writing the equation pointwise
(recall~\eqref{PAC1} and~\eqref{PAC2}).
A viscosity theory is also possible by appropriate modifications of the setting
(in particular, to pursue a viscosity theory, to be consistent with the elliptic framework,
one would need the additional assumption that the kernel is nonnegative).
For instance, for fractional elliptic equations a viscosity approach is useful to
establish existence results by the Perron method, which combined with
fractional elliptic regularity theory for viscosity solutions often provides the existence
of nice solutions for the Dirichlet problem (see e.g.~\cite{MR2707618}).
The viscosity setting will be briefly discussed in Section~\ref{VIS-PJN-c}.
}\end{remark}

\begin{remark} \label{higherorder}{\rm
{F}rom Definition \ref{def} one immediately sees that
for all~$j\in\N$ and~$K\in {\mathcal{K}}_{m,\vartheta}\cap
{\mathcal{K}}_{m+j,\vartheta}$,
\begin{equation*} \label{higherorder}
\mbox{if } Au\stackrel{m}{=}f, \mbox{ then } Au\stackrel{m+j}{=}f,
\end{equation*}
in $B_1$, since polynomials of degree at
most $m-1$ are also polynomials of degree at most $m+j-1$.
}\end{remark}

\begin{remark} \label{ifsmooth}
{\rm {F}rom Definition \ref{def} and Corollary~\ref{corlim} (used
here with $\tau_R:=\chi_R$, in the notation of \eqref{chi}),
we can write $Au\stackrel{m}{=}f_u$ in $B_1$ for any~$K\in \mathcal{K}_{m,\vartheta}$ and~$u \in \mathcal{C}_{\vartheta,K}$.
}\end{remark}

\begin{remark} \label{plusP}{\rm We
observe that from Definition \ref{def} it follows that any polynomial of degree less than
or equal to $m-1$ can be arbitrarily added to $f_R$
and subtracted from $P_R$ in \eqref{decomp},
hence, for any polynomial~$P$ with $\deg P \le m-1$ we have that
\begin{equation*}
\mbox{if } Au\stackrel{m}{=}f, \mbox{ then } Au\stackrel{m}{=}f+P 
\end{equation*}
in $B_1$.
}\end{remark}

We now investigate in further detail the convergence properties of the approximating source term~$f_R$.

\begin{lemma} \label{cor37}
Let $m \in \mathbb{N}_0$, $\vartheta \in [0,2]$
and $K\in {\mathcal{K}}_{m,\vartheta}$.
Let~$u \in \mathcal{C}_{\vartheta,K}$, $f$ and~$f_R$
be as in Definition~\ref{def}.

Then, if $R'>R>4$ we have that
\begin{equation} \label{37}
\inf\|f_{R'}-f_R-P\|_{L^\infty(B_1)}\le \int_{B_R^c} |u(y)|\sup_{{ |\alpha|=m}\atop{x\in B_1}}|\partial^\alpha_x K(x,y)|\,dy,
\end{equation}
where the $\inf$ in~\eqref{37}
is taken over all the polynomials $P$ with degree at most $m-1$.
\end{lemma}

\begin{proof} We define $v:=(1-\chi_4)u$. In this way $v=0$ in $B_4$ and $|v|\le|u|$, so
\begin{equation} \label{pos}
v \in \mathcal{C}_{\vartheta,K}.\end{equation}
Moreover, if $R>4$,
$$(\chi_R-\chi_4)u=(\chi_R-\chi_4)v.$$
Hence, from \eqref{decomp}, 
\begin{equation} \label{318}
A((\chi_R-\chi_4)v)=A((\chi_R-\chi_4)u)=f_R-f_4+P_R-P_4=f_R-f_4+\tilde P_R,
\end{equation}
where $\tilde P_R:=P_R-P_4$ is a polynomial of degree at most $m-1$.

We also remark that, due to~\eqref{pos},
we can use Theorem \ref{cut-off} here
on the function $v$.
More specifically, using Theorem \ref{cut-off}
on the function $v$
(twice, once with $\tau:=\chi_R$ and once with~$\tau:=\chi_4$),
we obtain that
\begin{equation} \label{318a}
\begin{split}
A((\chi_R-\chi_4)v)&=P_{v,\chi_R}-P_{v,\chi_4}+f_{v,\chi_R}-f_{v,\chi_4} \\
&=\bar P_{v,\chi_R}+(f_{1,v}+f_{2,v}+
f^*_{v,\chi_R})-(f_{1,v}+f_{2,v}+f^*_{v,\chi_4})\\&= \bar P_{v,\chi_R}+f^*_{v,\chi_R}-f^*_{v,\chi_4} \\
&=\bar P_{v,\chi_R}+\int_{B_R\setminus B_4} u(y)\psi(x,y)\,dy
\end{split}
\end{equation}
in $B_1$, where $\bar P_{v,\chi_R}:=P_{v,\chi_R}- P_{v,\chi_4}$ is a polynomial of degree at most $m-1$.
Comparing the right hand sides of \eqref{318} and \eqref{318a}, we obtain that in $B_1$
$$f_R=f_4+P^*_R+\int_{B_R\setminus B_3} u(y)\psi(x,y)\,dy,$$
where $P_R^*:=\bar P_{v,\chi_R}-\tilde P_R$ is a polynomial of
degree at most $m-1$.

Therefore, for any $R'>R$,
\begin{eqnarray*} f_{R'}-P^*_{R'}-f_R+P^*_R&=&
\left(
f_4+\int_{B_{R'}\setminus B_3} u(y)\psi(x,y)\,dy\right)
-
\left(
f_4+\int_{B_R\setminus B_3} u(y)\psi(x,y)\,dy\right)\\& =&
\int_{B_{R'}\setminus B_R} u(y)\psi(x,y)\,dy
\end{eqnarray*}
and, as a consequence,
\begin{equation} \label{319}
\|f_{R'}-P^*_{R'}-f_R+P^*_R\|_{L^\infty(B_1)}=\|\Psi_{R',R} \|_{L^\infty(B_1)},
\end{equation}
where
$$\Psi_{R',R}(x):=\int_{B_{R'}\setminus B_R} u(y)\psi(x,y)\,dy.$$

{F}rom \eqref{psibound} and Remark \ref{mcond}, we know that
\begin{equation*}
\begin{split}
&\|\Psi_{R',R}\|_{L^\infty(B_1)}\le \sup_{x \in B_1}\int_{B_{R'}\setminus B_R}|u(y)||\psi(x,y)|\,dy \\
&\qquad\qquad
\le  \int_{B_{R'}\setminus B_R}|u(y)|\sup_{{ |\alpha|=m}\atop{x\in B_1}
}|\partial^\alpha_x K(x,y)|\,dy\le  \int_{ B^c_R}|u(y)|\sup_{{ |\alpha|=m}\atop{x\in B_1}}|\partial^\alpha_x K(x,y)|\,dy.
\end{split}
\end{equation*}
This and \eqref{319} imply that
$$\|f_{R'}-P^*_{R'}-f_R+P^*_R\|_{L^\infty(B_1)}\le \int_{ B^c_R}|u(y)|\sup_{{ |\alpha|=m}\atop{x\in B_1}}|\partial^\alpha_x K(x,y)|\,dy,$$
which gives \eqref{37}.
\end{proof}

Next result deals with the stability of the equation under uniform convergence
(and this can be seen as an adaptation to
our setting of the result contained e.g. in Lemma~5 of~\cite{CS11}). 

\begin{lemma}\label{lemCS}
Let $\vartheta \in [0,2]$
and~$K\in{\mathcal{K}}_{0,\vartheta}$. 
For every~$k\in\N$, let~$u_k \in \mathcal{C}_{\vartheta,K}$ and~$f_k$ be bounded and continuous
in $B_1$. Assume that
\begin{equation}\label{CONUK0}
A u_k=f_k
\end{equation}
in~$B_1$, that
\begin{equation}\label{CONUK11}
{\mbox{$f_k$ converges uniformly in~$B_1$ to some function~$f$ as~$k\to+\infty$,}}
\end{equation}
that
\begin{equation*}\begin{split}&
{\mbox{$u_k$ converges in~$B_4$ to some function~$u$ as~$k\to+\infty$}}\\ &
{\mbox{in the topology of }}
\begin{cases}
L^\infty(B_4) &{\mbox{ if }} \vartheta=0,\\
C^\vartheta(B_4) &{\mbox{ if }} \vartheta\in(0,1),\\
C^{0,1}(B_4) &{\mbox{ if }} \vartheta=1,\\
C^{1,\vartheta-1}(B_4) &{\mbox{ if }} \vartheta\in(1,2),\\
C^{1,1}(B_4) &{\mbox{ if }} \vartheta=2.
\end{cases}\end{split}
\end{equation*}
and that\footnote{We observe that condition~\eqref{CONUK2} cannot be dropped from Lemma~\ref{lemCS}.
Indeed, if~$s\in(0,1)$ and
$$ \R\ni x\mapsto u_k(x):=-\frac{\chi_{(k,k^2)}(x)\,x^{2s}}{\log k}$$
we have that~$u_k\to0=:u$ locally uniformly and
that, for each~$x\in (-1,1)$,
\begin{eqnarray*}&&
\int_\R \frac{u_k(x)-u_k(y)}{|x-y|^{1+2s}}\,dy=\frac1{\log k}
\int_k^{k^2} \frac{y^{2s}}{(y-x)^{1+2s}}\,dy=:f_k(x).
\end{eqnarray*}
We stress that, if~$x\in(-1,1)$ and~$y>k$,
$$ y-x\ge y-1=\frac{k-1}k \,y +\frac{y}k-1\ge\frac{k-1}k \,y$$
and
$$ y-x\le y+1=\frac{k+1}k \,y-\frac{y}k+1\le\frac{k+1}k \,y.$$
As a result, if~$x\in(-1,1)$,
$$ f_k(x)\le\frac{k^{1+2s}}{(k-1)^{1+2s}}\,
\frac1{\log k}
\int_k^{k^2} \frac{y^{2s}}{y^{1+2s}}\,dy=\frac{k^{1+2s}}{(k-1)^{1+2s}}
$$
and
$$ f_k(x)\ge\frac{k^{1+2s}}{(k+1)^{1+2s}}\,
\frac1{\log k}
\int_k^{k^2} \frac{y^{2s}}{y^{1+2s}}\,dy=\frac{k^{1+2s}}{(k+1)^{1+2s}},
$$thus~$f_k\to1=:f$ uniformly in~$(-1,1)$. This example shows that
$$ (-\Delta)^s u_k=f_k\to f= 1\not=0=(-\Delta)^s u.$$
}
\begin{equation}\label{CONUK2}
\lim_{k\to+\infty} \int_{\mathbb{R}^n\setminus B_3}\big|u(y)-u_k(y)
\big||K(x,y)|\,dy=0,
\end{equation}
for every~$x\in B_1$.

Then,
\begin{equation*}
A u=f
\end{equation*}
in~$B_1$.
\end{lemma}

\begin{proof}
Let~$x_0\in B_1$ and $\rho>0$ such that~$B_\rho(x_0)\Subset B_1$.
We claim that
\begin{equation}\label{ST:001}
\lim_{k\to+\infty}
\int_{B_\rho(x_0)}\big(u_k(x_0)-u_k(y)\big)K(x_0,y)\,dy
=\int_{B_\rho(x_0)}\big(u(x_0)-u(y)\big)K(x_0,y)\,dy
.\end{equation}
For this, we distinguish two cases. If~$\vartheta\in[0,1]$, we observe that
\begin{eqnarray*}&&
\left|\int_{B_\rho(x_0)}\big(u_k(x_0)-u_k(y)\big)K(x_0,y)\,dy
-\int_{B_\rho(x_0)}\big(u(x_0)-u(y)\big)K(x_0,y)\,dy\right|\\&\le&
\int_{B_\rho(x_0)}\big|(u_k-u)(x_0)-(u_k-u)(y)\big|\, |K(x_0,y)|\,dy\\
&\le&\|u_k-u\|_{C^\vartheta(B_4)}
\int_{B_\rho(x_0)}|x_0-y|^\vartheta\, |K(x_0,y)|\,dy\\
&\le&C\,\|u_k-u\|_{C^\vartheta(B_4)}
\end{eqnarray*}
for some~$C>0$, thanks to~\eqref{locint}, and this proves~\eqref{ST:001}
in this case.

Hence, to complete the proof of~\eqref{ST:001}, we now assume that~$\vartheta\in(1,2]$.
In this case, we recall~\eqref{assym} and we see that, for $k$ sufficiently large,
\begin{eqnarray*}
&& \int_{B_{\rho}(x_0)}\big(u_k(x_0)-u_k(y)\big)K(x_0,y)\,dy\\
&=&\frac12 \int_{B_{\rho}}\big(u_k(x_0)-u_k(x_0+z)\big)K(x_0,x_0+z)\,dz
+\frac12\int_{B_{\rho}}\big(u_k(x_0)-u_k(x_0-z)\big)K(x_0,x_0-z)\,dz\\
&=&\frac12 \int_{B_{\rho}}\big(2u_k(x_0)-u_k(x_0+z)-u_k(x_0-z)\big)K(x_0,x_0+z)\,dz,
\end{eqnarray*}
and a similar computation holds with~$u$ instead of~$u_k$. Consequently,
recalling also~\eqref{siw74b8437654693} (used here with~$u_k-u$ in place of~$u$),
\begin{eqnarray*}&&
\left|\int_{B_\rho(x_0)}\big(u_k(x_0)-u_k(y)\big)K(x_0,y)\,dy
-\int_{B_\rho(x_0)}\big(u(x_0)-u(y)\big)K(x_0,y)\,dy\right|\\&\le&
\frac12 \int_{B_{\rho}}\big|2(u_k-u)(x_0)-(u_k-u)(x_0+z)-(u_k-u)(x_0-z)\big|\,|K(x_0,x_0+z)|\,dz\\&\le&
\frac{\|u_k-u\|_{C^{1,\vartheta-1}(B_4)}}2 
\int_{B_{\rho}}|z|^{\vartheta}\,|K(x_0,x_0+z)|\,dz\\
&\le& C\,\|u_k-u\|_{C^{1,\vartheta-1}(B_4)},
\end{eqnarray*}
for some~$C>0$, thanks to~\eqref{locint}, and this completes the proof of~\eqref{ST:001}.

We now claim that
\begin{equation}\label{ST:0012}
\lim_{k\to+\infty}
\int_{B_3\setminus B_\rho(x_0)}\big(u_k(x_0)-u_k(y)\big)K(x_0,y)\,dy
=\int_{B_3\setminus B_\rho(x_0)}\big(u(x_0)-u(y)\big)K(x_0,y)\,dy
.\end{equation}
To prove it, we use~\eqref{locint} to conclude that
\begin{eqnarray*}&&
\left|\int_{B_3\setminus B_\rho(x_0)}\big(u_k(x_0)-u_k(y)\big)K(x_0,y)\,dy
-\int_{B_3\setminus B_\rho(x_0)}\big(u(x_0)-u(y)\big)K(x_0,y)\,dy\right|\\
&\le&C\,\|u_k-u\|_{L^\infty(B_4)}
\int_{B_3\setminus B_\rho(x_0)} |K(x_0,y)|\,dy\\
&\le&C\,\|u_k-u\|_{L^\infty(B_4)}
\end{eqnarray*}
up to renaming~$C>0$ from line to line, and this establishes~\eqref{ST:0012}.

Furthermore, using~\eqref{locint},
\begin{eqnarray*}&&
\left|\int_{\R^n\setminus B_3}\big(u_k(x_0)-u_k(y)\big)K(x_0,y)\,dy
-\int_{\R^n\setminus B_3}\big(u(x_0)-u(y)\big)K(x_0,y)\,dy\right|\\
&\le&\int_{\R^n\setminus B_3}\big|u_k(x_0)-u(x_0)\big|\,|K(x_0,y)|\,dy+
\int_{\R^n\setminus B_3}\big|u_k(y)-u(y)\big||K(x_0,y)|\,dy\\&\le&
C\,\|u_k-u\|_{L^\infty(B_4)}+\int_{\R^n\setminus B_3}\big|u_k(y)-u(y)\big||K(x_0,y)|\,dy,
\end{eqnarray*}
which is infinitesimal thanks to~\eqref{CONUK2}.

Gathering together this, \eqref{ST:001} and~\eqref{ST:0012}, we conclude that~$Au_k(x_0)\to Au(x_0)$
as~$k\to+\infty$. {F}rom this, \eqref{CONUK0} and~\eqref{CONUK11} we obtain the desired result.
\end{proof}

A natural question is whether the stability result in Lemma~\ref{lemCS}
carries over directly to the setting introduced
in Definition~\ref{def}. The answer is in general negative,
as pointed out by the following counterexample:

\begin{proposition}\label{KMS:CONTO}
Let $k\in\N$. Let
$$ u_k(x):=\begin{cases} 0 & {\mbox{ if }}x\in(-\infty,k],\\
kx& {\mbox{ if }}x\in(k,+\infty)\end{cases}$$
and
$$ f_k(x):=
\frac{kx}{k - x}
+k\log\frac{k}{k-x}
.$$
Then,
\begin{eqnarray}
\label{56-09-01}&& \sqrt{-\Delta}u_k\stackrel{1}{=}f_k {\mbox{ in }}(-1,1),\\
\label{56-09-02}&&{\mbox{$u_k$
converges to zero locally uniformly,}} \\
\label{56-09-04}&&\lim_{k\to+\infty}\int_{\R\setminus(-1,1)}
\frac{|u_k(y)|}{|y|^{2+a}}\,dy=0\,{\mbox{ for all }}a>1,\\
\label{56-09-03}&&{\mbox{$f_k(x)$ converges to~$2x$ uniformly in $[-1,1]$.}}
\end{eqnarray}
\end{proposition}

\begin{proof}
Let~$R>k$ and~$u_{k,R}(x):=u_k(x)\chi_{(-R,R)}(x)$. Then, if~$x\in(-1,1)$,
\begin{eqnarray*}
&&\int_\R\frac{u_{k,R}(x)-u_{k,R}(y)}{|x-y|^2}\,dy
=k\int_k^{R}\frac{y}{(y-x)^2}\,dy=
\frac{kx}{k - x}-\frac{kx}{R-x}
+ k\log(R-x)-k\log(k-x) 
\\&&\qquad=
\frac{kx}{k - x}-\frac{kx}{R-x}
+ k\log\frac{R-x}R+k\log R+k\log\frac{k}{k-x}-k\log k \\&&\qquad=
f_k(x)
+ k\log\frac{R-x}R
-\frac{kx}{R-x}
+k\log R-k\log k.\end{eqnarray*}
Since the term~$k\log R-k\log k$ is a constant in~$x$ (hence a polynomial of degree zero)
and the function~$ k\log\frac{R-x}R
-\frac{kx}{R-x}$ goes to zero as~$R\to+\infty$,
the identity above proves~\eqref{56-09-01}.

Additionally, the claim in~\eqref{56-09-02} is obvious
and, if~$a>1$,
\begin{eqnarray*}&&
\lim_{k\to+\infty}\int_{\R\setminus(-1,1)}
\frac{|u_k(y)|}{|y|^{2+a}}\,dy=\lim_{k\to+\infty}k\int_{k}^{+\infty}
\frac{dy}{y^{1+a}}=\frac{1}{a}\lim_{k\to+\infty} \frac1{k^{a-1}}=0,
\end{eqnarray*}
thus establishing~\eqref{56-09-04}.

Furthermore, for every~$x\in[-1,1]$, if~$k$ is large enough,
\begin{eqnarray*}&&
|f_k(x)-2x|\le
\left|\frac{kx}{k - x}-x\right|
+\left|k\log\frac{k}{k-x}-x\right|
=\left|\frac{x^2}{k-x}\right|
+\left|k\int_1^{\frac{k}{k-x}} \frac{dt}t-x\right|\\&&\qquad\le
\frac{1}{k-1}
+\left|k\int_1^{1+\frac{x}{k-x}} \frac{dt}t
-\frac{kx}{k-x}\right|+\left|\frac{kx}{k-x}
-x\right|\\&&\qquad\le\frac{2}{k-1}
+k\,\left|\int_1^{1+\frac{x}{k-x}} \frac{dt}t
-\int_1^{1+\frac{x}{k-x}}dt\right|\le\frac{2}{k-1}
+k\,\int_{1-\frac{1}{k-1}}^{1+\frac{1}{k-1}} \frac{|1-t|}t\,dt\\&&\qquad\le\frac{6}{k-1},
\end{eqnarray*}
which gives~\eqref{56-09-03}.\end{proof}

Concerning the example in
Proposition~\ref{KMS:CONTO}, notice in particular that,
in~$(-1,1)$,
$$ \sqrt{-\Delta}u_k\stackrel{1}{=} f_k
\stackrel{\mbox{{\tiny{$k\to+\infty$}}}}{-\!\!\!-\!\!\!-\!\!\!\longrightarrow} 2x\stackrel{1}{\ne} 0=\sqrt{-\Delta}0,$$
showing that some care is necessary to pass Definition~\ref{def}
to the limit and additional assumptions are needed
to exchange the order in which different limits are taken.
\medskip

{F}rom the positive side, as an affirmative counterpart
of the counterexample in Proposition~\ref{KMS:CONTO},
we provide the following stability result
for the setting of Definition~\ref{def}:

\begin{proposition}\label{forthprop}
Let~$m\in\N_0$, $\vartheta \in [0,2]$
and~$K\in{\mathcal{K}}_{m,\vartheta}$.
For every~$k\in\N$, let~$u_k \in \mathcal{C}_{\vartheta,K}$
and~$f_k$ be bounded and continuous
in $B_1$. Assume that
\begin{equation}\label{SMDDIMD}
A u_k\stackrel{m}{=}f_k\quad{\mbox{ in~$B_1$,}}
\end{equation}
that
\begin{equation*}
{\mbox{$f_k$ converges uniformly in~$B_1$ to some function~$f$ as~$k\to+\infty$,}}
\end{equation*}
that
\begin{equation}\label{addfortwehn0000}\begin{split}&
{\mbox{$u_k$ converges in~$B_4$ to some function~$u$ as~$k\to+\infty$}}\\ &
{\mbox{in the topology of }}
\begin{cases}
L^\infty(B_4) &{\mbox{ if }} \vartheta=0,\\
C^\vartheta(B_4) &{\mbox{ if }} \vartheta\in(0,1),\\
C^{0,1}(B_4) &{\mbox{ if }} \vartheta=1,\\
C^{1,\vartheta-1}(B_4) &{\mbox{ if }} \vartheta\in(1,2),\\
C^{1,1}(B_4) &{\mbox{ if }} \vartheta=2,
\end{cases}\end{split}
\end{equation}
that
\begin{equation}\label{UHNBS D90PKSMD098765490}
\lim_{k\to+\infty}
\sup_{{x\in B_1}\atop{R>4}}
\int_{B_R\setminus B_1(x)}\big((u-u_k)(x)-(u-u_k)(y)\big)K(x,y)\,dy=0
\end{equation}
and that
\begin{equation}\label{90-109x4t23-2P}
\lim_{R\to+\infty}
\sup_{k\in\N}\int_{\R^n\setminus B_{R}}
|u_k(y)|\sup_{{ |\eta|=m}\atop{ z\in B_1}}|\partial^\eta_x K(z,y)|\,dy
=0.
\end{equation}
Then,
\begin{equation}\label{SKDMCFRDFV FM9}
A u\stackrel{m}{=}f\quad{\mbox{ in~$B_1$.}}
\end{equation}
\end{proposition}

To prove Proposition~\ref{forthprop}, we
establish a uniqueness result in the spirit of Lemma 1.2 of~\cite{MR3988080}:

\begin{lemma}\label{lemma:uniq}
Let~$m\in\N_0$, $\vartheta \in [0,2]$, $K\in{\mathcal{K}}_{m,\vartheta}$ and~$u \in \mathcal{C}_{\vartheta,K}$.
Let~$f_1$ and~$f_2$ be bounded and continuous
in $B_1$. Suppose that
\begin{equation}\label{pfi456u48674}
{\mbox{$A u\stackrel{m}{=}f_1$ and~$A u\stackrel{m}{=}f_2$ in~$B_1$.}}\end{equation}

Then, there exists a polynomial of degree at most~$m-1$ such that~$f_1-f_2=P$.
\end{lemma}

\begin{proof} In light of~\eqref{pfi456u48674} and Definition~\ref{def}, we have that
there exist two families of polynomials $P_R^1$ and~$P^2_R$, with degree at most~$ m-1$,
such that, for every~$x\in B_1$,
\begin{eqnarray*}
&&
\lim_{R \to +\infty} \big(A(\chi_Ru)(x)-P_R^1(x)\big)=f_1(x)\\{\mbox{and }}&&
\lim_{R \to +\infty} \big(A(\chi_Ru)(x)-P_R^2(x)\big)=f_2(x).
\end{eqnarray*}
As a consequence, for every~$x\in B_1$,
\begin{eqnarray*}&&
f_1(x)-f_2(x)=\lim_{R \to +\infty}\big(A(\chi_Ru)(x)-P_R^1(x)\big)-
\lim_{R \to +\infty} \big(A(\chi_Ru)(x)-P_R^2(x)\big)\\&&\qquad\qquad=
\lim_{R \to +\infty}\big(P_R^2(x)-P_R^1(x)\big).
\end{eqnarray*}
We remark that~$P_R^2-P_R^1$ is a polynomial of degree at most~$m-1$. Accordingly, we can use
Lemma~2.1 of~\cite{MR3988080} to conclude that~$f_1-f_2$ is a polynomial of degree at most~$m-1$.
This establishes the desired result.
\end{proof}

\begin{proof} [Proof of Proposition~\ref{forthprop}]
We exploit the setting of Corollary~\ref{corlim}
with~$\tau_R:=\chi_R$. In this way, for each~$k$, we
find a function $f_{u_k}: \mathbb{R}^n\to \mathbb{R}$
and a family of polynomials $P_{u_k,R}$, which have
degree at most $m-1$, such that, in~$B_1$,
\begin{equation}\label{02oek10ols}
\lim_{R\to +\infty}[A(\chi_R u_k)(x)-P_{u_k,R}(x)]=f_{u_k}(x).\end{equation}
As a matter of fact
(recall the
footnote on page~\pageref{FPTUNIF}),
we see that, for every~$x\in B_1$ and every~$k\in\N$,
\begin{equation}\label{VINDFKsdGgndfIR}
\Big| A(\chi_R u_k)(x)-P_{u_k,R}(x)-f_{u_k}(x)\Big|
\leq
C\int_{\R^n\setminus B_{R}}
|u_k(y)|\sup_{{ |\eta|=m}\atop{ z\in B_1} }|\partial^\eta_x K(z,y)|
\le\frac1R,
\end{equation}
as long as~$R$ is sufficiently large,
thanks to~\eqref{90-109x4t23-2P}.

Comparing~\eqref{02oek10ols} with Definition~\ref{def}, we thus conclude that~$
A u_k\stackrel{m}{=}f_{u_k}$ in~$B_1$.
This and~\eqref{SMDDIMD},
together with the uniqueness result in Lemma~\ref{lemma:uniq},
yield that~$f_{u_k}=f_k+\tilde P_k$ for a suitable polynomial~$\tilde P_k$
of degree at most~$m-1$.

Consequently, setting~$\tilde P_{u_k,R}:=P_{u_k,R}+\tilde P_k$,
we have that, by~\eqref{VINDFKsdGgndfIR},
\begin{equation}\label{2VINDFKsdGgndfIR}
\Big| A(\chi_R u_k)(x)-\tilde P_{u_k,R}(x)-f_k(x)\Big|
\le\frac1R.
\end{equation}

Now, we claim that, given~$R>4$,
\begin{equation}\label{MTHIN67G8LY9-98tr}
\lim_{k\to +\infty}
\sup_{B_1}\Big| A\big((u-u_{k})\chi_R\big)\Big|=0.
\end{equation}
To this end, 
we calculate that, for each~$x\in B_1$,
\begin{eqnarray*}&&
\Big| A\big((u-u_{k})\chi_R\big)(x)\Big|\\&=&\left|\int_{B_1(x)}\big((u-u_k)(x)-(u-u_k)(y)\big)K(x,y)\,dy+
\int_{B_R\setminus B_1(x)}\big((u-u_k)(x)-(u-u_k)(y)\big)K(x,y)\,dy\right.
\\&&\qquad\qquad\left.+\int_{\R^n\setminus B_R}(u-u_k)(x)\,K(x,y)\,dy\right|\\&\le&
\left|\int_{B_1(x)}\big((u-u_k)(x)-(u-u_k)(y)\big)K(x,y)\,dy\right|
+
\left|
\int_{B_R\setminus B_1(x)}\big((u-u_k)(x)-(u-u_k)(y)\big)K(x,y)\,dy\right|\\&&\qquad\qquad+\|u-u_k\|_{L^\infty(B_1)}
\int_{\R^n\setminus B_R}|K(x,y)|\,dy
\end{eqnarray*}
and hence~\eqref{MTHIN67G8LY9-98tr}
follows from~\eqref{locint}, \eqref{ST:001} (used here with~$\rho:=1$; notice that we can use~\eqref{ST:001}
in this setting
in light of~\eqref{addfortwehn0000})
and~\eqref{UHNBS D90PKSMD098765490}.

Thus, in light of~\eqref{MTHIN67G8LY9-98tr}, given~$R>4$
we can find~$k_R\in\N$ such that
\begin{equation}\label{RGHBNDINDFJDEBGIUJMS} \sup_{B_1}\Big| A\big((u-u_{k_R})\chi_R\big)\Big|\le\frac1R.\end{equation}
We define~$f_R:=f_{k_R}$ and~$P_R:=\tilde P_{u_{k_R},R}$.
We stress that~$P_R$ is a polynomial
of degree at most~$m-1$. Moreover, for every~$x\in B_1$,
\begin{eqnarray*}&&
\Big|A(\chi_R u)(x)-f_R(x)-P_R(x)\Big|\le
\Big| A\big((u-u_{k_R})\chi_R\big)(x)\Big|+
\Big|A(\chi_R u_{k_R})(x)-f_R(x)-P_R(x)\Big|\\&&\qquad\le\frac1R+
\Big|A(\chi_R u_{k_R})(x)-f_{k_R}(x)-\tilde P_{k_R}(x)\Big|\le\frac2R
\end{eqnarray*}
thanks to~\eqref{VINDFKsdGgndfIR}
and~\eqref{RGHBNDINDFJDEBGIUJMS},
which proves~\eqref{SKDMCFRDFV FM9}.
\end{proof}

A consequence of Lemmata~\ref{cor37} and~\ref{lemCS} is the following equivalence result:

\begin{corollary} \label{corequiv}
Let $\vartheta \in [0,2]$,
$K\in{\mathcal{K}}_{0,\vartheta}$ and~$u \in \mathcal{C}_{\vartheta,K}$. Let $f$ be bounded and continuous in $B_1$.

Then
$$Au=f {\mbox{ in }}B_1 $$
is equivalent to
$$Au\stackrel{0}{=}f \mbox{ in the sense of Definition \ref{def}}.$$
\end{corollary}
\begin{proof}
Suppose that $Au=f$ in $B_1$. Then, for $R>10$,
\begin{equation} \label{322}
A(\chi_{R/2}u)(x)=
Au(x)-A((1-\chi_{R/2})u)(x)
=f(x)+\int_{\mathbb{R}^n}(1-\chi_{R/2}(y))u(y)K(x,y)\,dy
\end{equation}
for every $x\in B_1$. Now, we set
$$w:=(\chi_R-\chi_{R/2})u.$$
We observe that~$w=0$ in $B_4$, so we can exploit
Theorem \ref{cut-off} to $w$ (applied here
with $m=0$) and get that, for any $x \in B_1$,
\begin{equation} \label{323}
\begin{split}
A&((\chi_R-\chi_{R/2})u)(x)=Aw(x) \\
&=f_{1,w}+f_{2,w}+f^*_{w,\chi_R}(x)=\int_{B_R\setminus B_3}w(y)\psi(x,y)\,dy=\int_{B_R\setminus B_{R/2}}u(y)\psi(x,y)\,dy.
\end{split}
\end{equation}
Hence, from \eqref{322} and \eqref{323}, we find that
\begin{equation}
\begin{split}
A(\chi_R u)(x)&=A((\chi_R-\chi_{R/2})u)(x)+A(\chi_{R/2}u)(x) \\
&=\int_{B_R\setminus B_{R/2}}u(y)\psi(x,y)\,dy+f(x)+\int_{\mathbb{R}^n}(1-\chi_{R/2}(y))u(y)\psi(x,y)\,
dy \\
&=:f_R(x)
\end{split}
\end{equation}
for every $x\in B_1$. We remark that $f_R\to f$ in $B_1$ as $R \to +\infty$, thanks to \eqref{psibound} (used here with $m=0$ and $\gamma =(0,\ldots,0)$) and \eqref{mcond}.

Now we recall Definition \ref{def} (here with $m=0$ and $P_R=0$)
and we conclude that $Au\stackrel{0}{=}f$ in $B_1$, as desired.

Conversely, we now suppose that $Au\stackrel{0}{=}f$ in $B_1$. {F}rom Definition \ref{def}
and the fact that $m=0$, we have that $P_R$ is identically zero, and so we can write that $A(\chi_R u)=f_R$ in $B_1$, with $f_R \to f$ in $B_1$ as $R\to +\infty$.
We observe that $\chi_R u$ approaches $u$ locally uniformly in $\mathbb{R}^n$. Also, we can use here Lemma \ref{cor37}: 
in this way, we find that
$$\|f_{R'}-f_R\|_{L^\infty(B_1)}\le \int_{B_R^c} |u(y)|\sup_{x\in B_1}|K(x,y)|\,dy.$$
Therefore, we send $R' \to +\infty$ and obtain that, for any $x \in B_1$,
$$|f(x)-f_R(x)|=\lim_{R'\to +\infty}|f_{R'}(x)-f_R(x)|\le
\lim_{R'\to +\infty}\|f_{R'}-f_R \|_{L^\infty(B_1)}\le \int_{B_R^c} |u(y)|\sup_{x\in B_1}
|K(x,y)|\,dy.$$
As a consequence, recalling~\eqref{mcond}
(here with~$m=0$) we have that $f_R$ converges to $f$ uniformly in $B_1$ as $R \to +\infty$.
{F}rom this, we can exploit Lemma~\ref{lemCS}
and conclude that $Au=f$, as desired.
\end{proof}

A natural question deals with the consistency of the operator setting for functions that
are sufficiently well-behaved to allow definitions related to two different indices: roughly
speaking, in the best possible scenario, if we know that~$Au\stackrel{m}{=}f$ and~$j\le m$,
can we say that~$Au\stackrel{j}{=}f$? Posed like this, the answer to this question is negative,
since, after all, in light of Lemma~\ref{lemma:uniq}, the function~$f$ is uniquely defined only ``up to a polynomial''.
Nevertheless, the answer becomes positive if we take into account this additional polynomial normalization,
as stated in the next result:

\begin{lemma} \label{lemmaj} 
Let $j, m \in \N_0$, with~$j\le m$, $\vartheta\in[0,2]$
and $K \in \mathcal{K}_{j,\vartheta}\cap \mathcal{K}_{m,\vartheta}$.
Let~$f$ be bounded and continuous in~$B_1$ and let~$u \in \mathcal{C}_{\vartheta,K}$
such that
\begin{equation}\label{d3ivbrgretgyt57t}
Au\stackrel{m}{=}f\end{equation} in $B_1$. 

Then, there exist a function $\bar f$ and a polynomial $P$ of degree at most $m-1$,
such that~$\bar f=f+P$ and~$Au\stackrel{j}{=}\bar f$ in $B_1$.
\end{lemma}

\begin{proof} Let $v:=(1-\chi_4)u$ and $w:=\chi_4 u$. We notice that~$v$,
$w\in \mathcal{C}_{\vartheta,K}$. Hence, since $K\in \mathcal{K}_{j,\vartheta}$, recalling
Remark~\ref{ifsmooth}, \eqref{f12*},
\eqref{repr2} and~\eqref{f3}, we can write that
$$Av\stackrel{j}{=}f_v=\int_{B_4^c} u(y)\psi (x,y)$$
in $B_1$. That is, by Definition \ref{def},
\begin{equation} \label{Achiv}
A(\chi_R v)=\int_{B_4^c} u(y)\psi (x,y)\,dy+\tilde\varphi_R+Q_R,
\end{equation}
for some $\tilde\varphi_R$ such that $\tilde\varphi_R \to 0$ in $B_1$ as $R \to +\infty$ and a polynomial $Q_R$ of degree at most $j-1$.

Furthermore, by~\eqref{d3ivbrgretgyt57t}, and recalling Definition \ref{def}, we get that 
\begin{equation} \label{Achiu}
A(\chi_R u)=f+\varphi_R+P_R,
\end{equation}
for some $\varphi_R$ such that $\varphi_R \to 0$ if $R \to +\infty$
and a polynomial $P_R$ with deg$P_R\le m-1$. Therefore,
subtracting \eqref{Achiv} from \eqref{Achiu}, we obtain 
\begin{equation}\label{poiurt3t3645}
f+\varphi_R+P_R-\int_{B_4^c} u(y)\psi (x,y)\,dy-\tilde\varphi_R-Q_R=A(\chi_R(u-v))=A(\chi_R w)
.\end{equation}

We notice that, for every~$x\in B_1$ and every~$R>4$,
\begin{eqnarray*}
&& A(\chi_R w)(x)=\int_{\R^n}(\chi_R w(x)-\chi_R w(y))\,K(x,y)\,dy\\&&\qquad\quad=
\int_{B_R}(w(x)-w(y))\,K(x,y)\,dy+\int_{\R^n\setminus B_R} w(x)\,K(x,y)\,dy\\&&\qquad\quad
=\int_{B_4}(w(x)-w(y))\,K(x,y)\,dy+\int_{\R^n\setminus B_4} w(x)\,K(x,y)\,dy\\ &&\qquad\quad
=\int_{\R^n}(\chi_4w(x)-\chi_4 w(y))\,K(x,y)\,dy=A(\chi_4 w)(x).
\end{eqnarray*}
As a consequence of this and~\eqref{poiurt3t3645}, we have that, in~$B_1$,
$$f+\varphi_R+P_R-\int_{B_4^c} u(y)\psi (x,y)\,dy-\tilde\varphi_R-Q_R=A(\chi_4 w).$$
This shows that the limit
$$\lim_{R \to +\infty} (\varphi_R+P_R-\tilde\varphi_R-Q_R)$$
exists. As a result, 
the limit $$\lim_{R \to +\infty}(P_R-Q_R)$$ exists.
Then, exploiting Lemma 2.1 in \cite{MR3988080} we conclude that
$$\lim_{R \to +\infty} (P_R-Q_R)=P,$$
for some polynomial $P$ of degree at most $m-1$.

Now we set $\bar f:=f+P$ and $S_R:=\varphi_R+P_R-Q_R-P$, and
we see that $S_R \to 0$ as $R \to +\infty$. Thus, from~\eqref{Achiu} we obtain that
$$A(\chi_R u)=\bar f+S_R+Q_R$$
in $B_1$. Since the degree of $Q_R$ is at most $j-1$, this shows that $Au\stackrel{j}{=}\bar f$ in $B_1$, as desired.
\end{proof}

\section{The Dirichlet problem}\label{SEC-2}

In this section we consider the existence problem
for equations involving general operators that are defined ``up to a polynomial''.
The main result is the following:

\begin{theorem}\label{KS:098iKS-904596}
Let $m \in \N_0$, $\vartheta \in [0,2]$, $K\in
\mathcal{K}_{0,\vartheta}\cap {\mathcal{K}}_{m,\vartheta}$
and~$u \in \mathcal{C}_{\vartheta,K}$. Assume that
$u_0\in L^1_{\rm loc}(B_1^c)$ satisfies~\eqref{ring} and~\eqref{mcond} and~$
f: B_1 \to \R$ is bounded and
continuous in $B_1$.

Additionally, suppose that
for any~$\tilde f: B_1 \to \R$ which is bounded and
continuous in $B_1$ and any~$\tilde u_0\in L^1(B_1^c)$
there exists a unique
solution~$\tilde u\in{\mathcal{C}}_\vartheta$ to the Dirichlet problem
\begin{equation}  \label{Dirstandard}
\begin{cases} 
A\tilde u=\tilde f &{\mbox{ in }}B_1,\\
\tilde u= \tilde u_0 &{\mbox{ in }} B_1^c.
\end{cases}
\end{equation}
Then, there exists a
function $u \in \mathcal{C}_{\vartheta,K}$ such that
\begin{equation}  \label{Dir}
\begin{cases} 
Au\stackrel{m}{=}f&{\mbox{ in }}B_1,\\
u=u_0&{\mbox{ in }}B_1^c.
\end{cases}
\end{equation}
Also, the solution to~\eqref{Dir} is not unique, since
the space of solutions of \eqref{Dir} has dimention $N_m$, with
\begin{equation} \label{soldim}
N_m:=
\sum_{j=0}^{m-1}
\binom{j+n-1}{n-1}.
\end{equation}
\end{theorem}

\begin{proof}
To begin with, we prove the existence of solutions for~\eqref{Dir}. To do this, we define 
\begin{equation*}
u_1:=\chi_{B_4^c}\,u_0 \qquad
{\mbox{ and }} \qquad\tilde u_0:=\chi_{B_4\setminus B_1} \,u_0.
\end{equation*}
Since $u_1$ vanishes in $B_4$ and
$K\in \mathcal{K}_{m,\vartheta}$, we can write $Au_1\stackrel{m}{=}f_{u_1}$ in $B_1$,
for some function $f_{u_1}$, due to Remark~\ref{ifsmooth}.

We now consider
the solution of \eqref{Dirstandard} with~$\tilde f:=f-f_{u_1}$. Therefore, using
Remark \ref{higherorder} and Corollary~\ref{corequiv}
we obtain that
\begin{equation*} 
\begin{cases} 
A\tilde u\stackrel{m}{=}\tilde f&{\mbox{ in }}B_1,\\
\tilde u= \tilde u_0&{\mbox{ in }}B_1^c.
\end{cases}
\end{equation*}
Then, we set $u:=u_1+\tilde u$ and we get that $Au=Au_1+A
\tilde u\stackrel{m}{=}f_{u_1}+\tilde f=f$ in $B_1$.
Moreover, we have that~$u=u_1+\tilde u_0=u_0$ in $B_1^c$, that is $u\in{\mathcal{C}}_{\vartheta,K}$
is solution of \eqref{Dir}. This
establishes the existence of solution for~\eqref{Dir}.

Now we prove that solutions of~\eqref{Dir}
are not unique and determine the dimension
of the corresponding linear space. For this, we notice that for any polynomial $P$ with $\deg P \le m-1$ there exists a unique solution $\tilde u_P\in{\mathcal{C}}_\vartheta$ of the problem
\begin{equation}\label{PIUUET}
\begin{cases} 
A\tilde u_P=P&{\mbox{ in }}B_1,\\
\tilde u_P=0&{\mbox{ in }} B_1^c,
\end{cases}
\end{equation}
due to the existence and uniqueness assumption
for~\eqref{Dirstandard}. This is equivalent to say that $A\tilde u_P\stackrel{0}{=}P$ in $B_1$, thanks to Corollary \ref{corequiv}. Using Remark \ref{higherorder}, we obtain that $A\tilde u_P\stackrel{m}{=}P$ in $B_1$. Thus, applying
Remark \ref{plusP}, we obtain that $\tilde u_P$ is a solution of
\begin{equation} \label{DirP} 
\begin{cases} 
A\tilde u_P\stackrel{m}{=}0&{\mbox{ in }}B_1,\\
\tilde u_P=0&{\mbox{ in }} B_1^c.
\end{cases}
\end{equation}
{F}rom this it follows that if $u$ is a solution of \eqref{Dir}, then $u+\tilde u_P$ is also a solution of \eqref{Dir}.

Viceversa, if $u$ and $v$ are two solutions of~\eqref{Dir}, then $w:=u-v$ is a solution of
\begin{equation*} 
\begin{cases} 
Aw\stackrel{m}{=}0&{\mbox{ in }}B_1,\\
w=0&{\mbox{ in }} B_1^c.
\end{cases}
\end{equation*}
Here we can apply Lemma \ref{lemmaj} with $j:=0$ thus obtaining that $Aw\stackrel{0}{=}P$ in $B_1$, where $P$ is a polynomial of $\deg P \le m-1$. Using again Corollary \ref{corequiv}, one deduces that
\begin{equation}\label{KSMD-DIV} 
\begin{cases} 
Aw=P&{\mbox{ in }}B_1,\\
w=0&{\mbox{ in }}B_1^c.
\end{cases}
\end{equation}
Therefore, the uniqueness of the solution of  \eqref{KSMD-DIV},
confronted with~\eqref{PIUUET}, gives us that~$w=\tilde u_P$, and thus~$v=u+\tilde u_P$.

This reasoning gives that the space of solutions of \eqref{Dir} is isomorphic to the space of polynomials with degree
less than or equal to $m-1$, which has exactly
dimension $N_m$, given by \eqref{soldim} (see e.g.~\cite{2021arXiv210107941D}).
\end{proof}

\section{A viscosity approach}\label{VIS-PJN-c}

Up to now, we focused our attention on the case of equations defined pointwise.
In principle, this requires functions that are ``sufficiently regular'' for the equation to be satisfied at every given point.
However, a less restrictive approach adopted in the classical theory of elliptic equations 
is to consider weaker notions of solutions (and possibly recover the pointwise setting via
an appropriate regularity theory): in this spirit, a convenient setting, which is also useful in case
of fully nonlinear equations, is that of viscosity solutions, 
which does not require a high degree of regularity of the solution itself
since the equation is computed pointwise only at smooth functions touching from either below or above
(see e.g.~\cite{CC95} for a thorough discussion on viscosity solutions).

In this section, we recast the setting of general operators defined ``up to a polynomial''
into the viscosity solution framework. To this end, we proceed as follows.
For all $m \in \N_0$
and~$\vartheta\in[0,2]$, we define $\mathcal{K}_{m,\vartheta}^+$ as the space of kernels $K=K(x,y)$ 
verifying~\eqref{taylor}, \eqref{locint} and~\eqref{assym}, and
such that 
\begin{equation}
\label{positive}K(x,y)\ge 0 {\mbox{ for all $x \in B_1$ and $y \in \R^n$}}.
\end{equation}
Given~$K\in\mathcal{K}_{m,\vartheta}^+$, we consider the space~$
{\mathcal{V}}_{K}$
of all the functions~$u\in
L^1_{\rm{loc}}(\R^n) \cap C(B_4)
\cap L^\infty(B_4)$ for which
\begin{eqnarray}
\label{ringv}
&&\sum_{|\alpha|\le m-1}\int_{B_R\setminus B_3}
|u(y)|\,|\partial^\alpha_x K(x,y)|\,dy<+\infty
\quad {\mbox{ for all~$R>3$ and $x \in B_1$}}\\
\label{mcondv}{\mbox{and }}
&&\int_{B^c_3} |u(y)|\sup_{{|\alpha|=m}\atop{x\in B_1}}
|\partial^\alpha_x K(x,y)|\,dy<+\infty.
\end{eqnarray}

\begin{remark} \label{implication}
{\rm Notice that if \eqref{positive} holds true and $K \in \mathcal{K}_{m,\vartheta}$ for some~$m\in\N_0$ and
some~$\vartheta \in [0,2]$, then~$K \in \mathcal{K}_{m,\vartheta}^+$.}
\end{remark}

In the viscosity framework we introduce the following definition.
\begin{definition} \label{defv} Let~$m\in\N_0$, $\vartheta\in[0,2]$,
$K\in \mathcal{K}_{m,\vartheta}^+$,
$u \in \mathcal{V}_{K}$ and $f : B_1 \to \R$ be bounded and continuous. 
We say that 
$$
Au\stackrel{m}{=}f \quad \mbox{in $B_1$ in the viscosity sense}
$$
if there exist a family of polynomials $P_R$, with $\deg P_R \le m-1$, and bounded
and continuous functions~$f_R: B_1 \to \R$ such that
\begin{equation} \label{decompv}
A(\chi_R u)=f_R+P_R
\end{equation}
in $B_1$ in the viscosity sense, with
\begin{equation} \label{limfv}
\lim_{R \to +\infty} f_R(x) =f(x) \quad{\mbox{ uniformly in }}B_1.
\end{equation}
\end{definition}

\begin{remark}{\rm
We point out that the limit in~\eqref{limfv}
is assumed to hold uniformly (this is a stronger assumption
than the one in~\eqref{limf} that was assumed
for the pointwise setting,
and it is taken here to make the setting compatible with the viscosity method,
see e.g. the proof of the forthcoming Corollary~\ref{corequivv}). See also~\cite{MR4038144}
for related observations.}\end{remark}

\begin{remark} \label{higherorderv}{\rm
We observe that,
for all~$j\in\N$ and~$K\in \mathcal{K}_{m,\vartheta}^+\cap
\mathcal{K}_{m+j,\vartheta}^+$,
\begin{equation*} 
\mbox{if } Au\stackrel{m}{=}f, \mbox{ then } Au\stackrel{m+j}{=}f
\end{equation*}
in $B_1$ in the viscosity sense of Definition \ref{defv}.
}\end{remark}

\begin{remark} \label{plusPv}{\rm {F}rom Definition \ref{defv} it follows that any polynomial of degree less than
or equal to $m-1$ can be arbitrarily added to $f_R$
and subtracted from $P_R$ in \eqref{decompv},
hence, for any polynomial~$P$ with $\deg P \le m-1$ we have that
\begin{equation*}
\mbox{if } Au\stackrel{m}{=}f, \mbox{ then } Au\stackrel{m}{=}f+P 
\end{equation*}
in $B_1$ in the viscosity sense of Definition \ref{defv}.
}\end{remark}

We now establish that when the structure is compatible with the both the settings in Definitions~\ref{def}
and~\ref{defv}, the pointwise and viscosity frameworks are equivalent:

\begin{lemma}\label{EQUIBVA}
Let~$m\in\N_0$,
$\vartheta \in [0,2]$, $K\in {\mathcal{K}}_{m,\vartheta}$,
$u \in \mathcal{C}_{\vartheta,K}$ and $f : B_1 \to \R$ be bounded and continuous. 
If~\eqref{positive} holds true and $u$ is a solution of $$Au\stackrel{m}{=}f \mbox{ in the
sense of Definition~\ref{def}},$$ then $K \in \mathcal{K}_{m,\vartheta}^+$, $u
\in \mathcal{V}_{K}$ and
it is a solution of $$Au\stackrel{m}{=}f \mbox{ in the viscosity sense of Definition~\ref{defv}}.$$

Conversely, 
let~$m\in\N_0$, $\vartheta\in[0,2]$,
$K\in \mathcal{K}_{m,\vartheta}^+$,
$u \in \mathcal{V}_{K}$ and $f : B_1 \to \R$ be bounded and continuous.
If~$K \in \mathcal{K}_{m,\vartheta}$
and $u\in{\mathcal{C}}_{\vartheta,K}$ is a solution of $$Au\stackrel{m}{=}f \mbox{ in the viscosity sense of Definition~\ref{defv}},$$
then~$u$ is a solution of 
$$Au\stackrel{m}{=}f \mbox{ in the sense of Definition~\ref{def}}.$$
\end{lemma}
\begin{proof}
Assume that $u$ is a solution of $Au\stackrel{m}{=}f \mbox{ in $B_1$ in the sense of Definition~\ref{def}}$. It follows that there exist a family of polynomials $P_R$ with $\deg P_R \le m-1$ and functions $f_R$ such that \begin{equation}\label{CJS-080}
A(\chi_R u)=f_R+P_R\end{equation} pointwise in $B_1$.
Since~$u \in \mathcal{C}_{\vartheta,K}$
we have that also~$u \in \mathcal{V}_{K}$. Moreover, we have that $K \in \mathcal{K}_{m,\vartheta}^+$, due to \eqref{positive} and Remark \ref{implication}. 

Now we observe that if~$v\in{\mathcal{C}}_\vartheta$ and vanishes outside~$B_2$, $g$ is a bounded and continuous function,
and~$Av=g$ pointwise in~$B_1$, then also
\begin{equation}\label{CJS-081}
{\mbox{$Av=g$ in~$B_1$ in the viscosity sense.}}\end{equation} To check this,
let~$\varphi$ be a smooth function touching~$v$ from below at some point~$x_0\in B_1$. Then, we have that~$v(y)-\varphi(y)\ge0= v(x_0)-\varphi(x_0)$ for all~$y\in\R^n$ and therefore, by~\eqref{positive},
$$ A\varphi(x_0)=\int_{\mathbb{R}^n}(\varphi(x_0)-\varphi(y))K(x_0,y)\,dy\ge
\int_{\mathbb{R}^n}(v(x_0)-v(y))K(x_0,y)\,dy=Av(x_0)=g(x_0).
$$
Similarly, if~$\varphi$ touches~$v$ from above, one obtains the opposite inequality, and these
observations complete the proof of~\eqref{CJS-081}.

As a consequence of~\eqref{CJS-080}
and~\eqref{CJS-081}, we obtain that
$$
A(\chi_R u)=f_R+P_R$$ in $B_1$ in the viscosity sense.

Hence, to finish the first part of the proof, we show that
\begin{equation}\label{ifefeugthut8t49}
{\mbox{$f_R\to f$ uniformly in~$B_1$.}}
\end{equation}
For this, we observe that, in light of Remark~\ref{ifsmooth}, we can write~$Au\stackrel{m}{=}f_u$ in $B_1$
in the sense of Definition~\ref{def}. That is, recalling Corollary~\ref{corlim},
there exist functions $f_{u,R}$ and a family of polynomials $P_{u,\tau_R}$, which have degree at most $m-1$, such that
\begin{equation}\label{CJS-0802}
A(\chi_R u)=f_{u,R}+P_{u,\chi_R}\end{equation} pointwise in~$B_1$.

Furthermore, using Lemma~\ref{cor37}, we have that, if $R'>R>4$, there exists a polynomial~$P_{R,R'}$
of degree at most~$m-1$ such that
\begin{equation}\label{04386jreighrjij}
\|f_{u,R'}-f_{u,R}-P_{R,R'}\|_{L^\infty(B_1)}\le \int_{B_R^c} |u(y)|\sup_{{|\alpha|=m}\atop{x\in B_1
}}|\partial^\alpha_x K(x,y)|\,dy.
\end{equation}
We now claim that
\begin{equation}\label{sjiery48674895}
P_{R,R'}=0.
\end{equation}
Indeed, by a careful inspection of the proof of Lemma~\ref{cor37}, one can notice that
the polynomial~$P_{R,R'}$ is explicit and, denoting by~$v:=(1-\chi_4)u$, it is equal to
$$
P_{v,\chi_{R'}}-P_{u,\chi_{R'}}-P_{v,\chi_R}+P_{u,\chi_R},
$$
where we have used the notation of Theorem~\ref{cut-off}.
In particular, recalling the notation in formulas~\eqref{theta} and~\eqref{P}, we have that
the coefficients of the polynomial~$P_{R,R'}$ are given by
\begin{eqnarray*}&&
\int_{B^c_3} (\chi_{R'} v)(y)\frac{\partial^{\alpha}_x K(0,y)}{\alpha!}\,dy
-\int_{B^c_3} (\chi_{R'} u)(y)\frac{\partial^{\alpha}_x K(0,y)}{\alpha!}\,dy\\&&\qquad
-\int_{B^c_3} (\chi_{R} v)(y)\frac{\partial^{\alpha}_x K(0,y)}{\alpha!}\,dy
+\int_{B^c_3} (\chi_{R} u)(y)\frac{\partial^{\alpha}_x K(0,y)}{\alpha!}\,dy\\&=&
\int_{B_{R'}\setminus B_4}  u(y)\frac{\partial^{\alpha}_x K(0,y)}{\alpha!}\,dy
-\int_{B_{R'}\setminus B_3} u(y)\frac{\partial^{\alpha}_x K(0,y)}{\alpha!}\,dy\\&&\qquad
-\int_{B_R\setminus B_4} u(y)\frac{\partial^{\alpha}_x K(0,y)}{\alpha!}\,dy
+\int_{B_R\setminus B_3}  u(y)\frac{\partial^{\alpha}_x K(0,y)}{\alpha!}\,dy
\\&=&-
\int_{B_{4}\setminus B_3}  u(y)\frac{\partial^{\alpha}_x K(0,y)}{\alpha!}\,dy
+\int_{B_4\setminus B_3} u(y)\frac{\partial^{\alpha}_x K(0,y)}{\alpha!}\,dy\\&=&0,
\end{eqnarray*}
for every~$|\alpha|\le m-1$, which proves~\eqref{sjiery48674895}.

Hence, using the information of formula~\eqref{sjiery48674895} into~\eqref{04386jreighrjij}, we obtain
that~$f_{u,R}$ converges to~$f_u$ uniformly in~$B_1$.

Now, as a consequence of~\eqref{CJS-080} and~\eqref{CJS-0802},
we have that
\begin{equation}\label{dtuy54yu4ygjl}
\lim_{R\to+\infty} \big( P_{u,\chi_R}-P_R\big)=
\lim_{R\to+\infty} \big(f_R-f_{u,R}\big)=f-f_u
\end{equation}
in~$B_1$. Therefore, in light of Lemma~2.1 in~\cite{MR3988080}, we have that
the convergence in~\eqref{dtuy54yu4ygjl} is uniform in~$B_1$. 
Furthermore,
\begin{eqnarray*}
\|f_R-f\|_{L^\infty(B_1)}\le \|f_R-f_{u,R}+f_u-f\|_{L^\infty(B_1)}+\|f_{u,R}-f_u\|_{L^\infty(B_1)}.
\end{eqnarray*}
These considerations prove~\eqref{ifefeugthut8t49} and therefore,
the first part of the proof is complete.

Now take $K \in \mathcal{K}_{m,\vartheta}^+$
and a solution $u\in{\mathcal{C}}_{\vartheta,K}$ to $Au\stackrel{m}{=}f \mbox{ in $B_1$ in the viscosity
sense of Definition~\ref{defv}}$. We have that there exist a family of
polynomials $P_R$ with $\deg P_R \le m-1$ and functions $f_R$ such that 
\begin{equation}\label{9876MNBV}
{\mbox{$A(\chi_R u)=f_R+P_R$ in $B_1$ in the
viscosity sense.}}\end{equation}
Our objective is now to check that
\begin{equation}\label{9876MNBV-2}
{\mbox{the equation in~\eqref{9876MNBV} holds true in the pointwise
sense as well.}}\end{equation}
Indeed, once this is established, we can send~$R\to+\infty$
and conclude that~$Au\stackrel{m}{=}f$ in the pointwise
sense of Definition~\ref{def}.
To prove~\eqref{9876MNBV-2}, we use a convolution argument.
We pick~$\rho\in(0,1)$ and we define~$v_\e$ to be the convolution
of~$\chi_R u$ against a given mollifier~$\eta_\e$.
We also denote by~$g_\e$ the convolution of~$f_R+P_R$
against~$\eta_\e$ and we remark that, if~$\e$ is small enough, then~$
A v_\e=g_\e$ in $B_\rho$ in the
viscosity sense, and actually also in the pointwise sense, since~$v_\e$
is smooth and can be used itself as a test function in the viscosity definition.
Hence, we can take any point~$x_0\in B_\rho$ and conclude that
\begin{equation}\label{0-0-AM 90-01}
\lim_{\e\to0}
\int_{\mathbb{R}^n}(v_\e(x_0)-v_\e(y))K(x_0,y)\,dy
=\lim_{\e\to0}A v_\e(x_0)=\lim_{\e\to0}g_\e(x_0)=f_R(x_0)+P_R(x_0).
\end{equation}
We now claim that, for all~$R>5$,
\begin{equation}\label{0-0-AM 90-02}
\lim_{\e\to0}
\int_{\mathbb{R}^n}(v_\e(x_0)-v_\e(y))K(x_0,y)\,dy
=
\int_{\mathbb{R}^n}(\chi_R(x_0)u(x_0)-\chi_R(y)u(y))K(x_0,y)\,dy.\end{equation}
We stress that, once this is proved, then~\eqref{9876MNBV-2}
would follow directly from~\eqref{0-0-AM 90-01}.
Hence, our goal now is to check~\eqref{0-0-AM 90-02}.
We perform the argument when~$\vartheta\in(1,2]$
(the argument when~$\vartheta\in[0,1]$ being similar
and simpler, not requiring any additional symmetrization).
We exploit~\eqref{assym} to see that
\begin{equation}\label{kjhgfFGHNSiujhbVS}
\begin{split}
&2\int_{B_1(x_0)}(v_\e(x_0)-v_\e(y))K(x_0,y)\,dy\\ &\qquad
=\int_{B_1}(v_\e(x_0)-v_\e(x_0+z))K(x_0,x_0+z)\,dz
+
\int_{B_1}(v_\e(x_0)-v_\e(x_0-z))K(x_0,x_0-z)\,dz\\&\qquad
=\int_{B_1}\big(2v_\e(x_0)-v_\e(x_0+z)-v_\e(x_0-z)\big)K(x_0,x_0+z)\,dz.
\end{split}\end{equation}
Also, since~$u\in{\mathcal{C}}_{\vartheta,K}$ (and we are supposing~$\vartheta\in(1,2]$),
for all~$z\in B_1$,
\begin{eqnarray*}&&
|2v_\e(x_0)-v_\e(x_0+z)-v_\e(x_0-z)|=
\left|\int_0^1\Big( \nabla v_\e(x_0+tz)
-\nabla v_\e(x_0-tz)\Big)\cdot z\,dt
\right|\\&&\qquad\le
|z|\int_0^1
\Big| \nabla v_\e(x_0+tz)
-\nabla v_\e(x_0-tz)\Big|\,dt\\&&\qquad=
|z|\int_0^1
\left| \int_{B_\e} \nabla (\chi_R u)(x_0+tz-\zeta)\eta_\e(\zeta)\,d\zeta
-\int_{B_\e} \nabla (\chi_R u)(x_0-tz-\zeta)\eta_\e(\zeta)\,d\zeta\right|\,dt
\\&&\qquad\le
|z|\int_0^1
\int_{B_\e} \left|\nabla (\chi_R u)(x_0+tz-\zeta)-
\nabla (\chi_R u)(x_0-tz-\zeta)\right|\eta_\e(\zeta)\,d\zeta\,dt
\\&&\qquad\le
C |z|^\vartheta\int_0^1
\int_{B_\e} \eta_\e(\zeta)\,d\zeta\,dt\\&&\qquad=C |z|^\vartheta,
\end{eqnarray*}
for some~$C>0$.

{F}rom this and the Dominated Convergence Theorem,
recalling~\eqref{locint}, we deduce from~\eqref{kjhgfFGHNSiujhbVS}
that
\begin{eqnarray*}&&
\lim_{\e\searrow0}2\int_{B_1(x_0)}(v_\e(x_0)-v_\e(y))K(x_0,y)\,dy
=\int_{B_1}\lim_{\e\searrow0}\big(2v_\e(x_0)-v_\e(x_0+z)-v_\e(x_0-z)\big)K(x_0,x_0+z)\,dz\\
&&\qquad=
\int_{B_1}\big(2(\chi_R u)(x_0)-(\chi_R u)(x_0+z)-(\chi_R u)(x_0-z)\big)K(x_0,x_0+z)\,dz
\\&&\qquad
=2\int_{B_1(x_0)}\big((\chi_R u)(x_0)-(\chi_R u)(y)\big)K(x_0,y)\,dy.
\end{eqnarray*}
Using again the Dominated Convergence Theorem,
one deduces~\eqref{0-0-AM 90-02} from the previous equation, as desired.
\end{proof}

Next result shows
the stability of the equation under the uniform convergence in the viscosity sense. To this end, we will also
assume other mild conditions on the kernel.
First of all, we assume a continuity hypothesis in the first variable,
that is we suppose that
\begin{equation}\label{CONZTINY}
{\mbox{for all~$y\in\R^n$ and all~$x_0\in B_3\setminus\{y\}$, }}\;
\lim_{x\to x_0} K(x,y)=K(x_0,y).
\end{equation}
Additionally, we assume
a local integrability condition outside a possible singularity of the kernel
and a locally uniform version of condition~\eqref{locint}, namely we suppose that
\begin{equation}\label{locint-RINF-PRE}{\mbox{for all $x_0\in B_3$ and all~$r>0$, }}\;
\int_{\R^n\setminus B_r(x_0)} \;\sup_{x\in B_3\setminus B_r(y)} |K(x,y)|\,dy<+\infty
\end{equation}
and
\begin{equation}\label{locint-RINF}
\int_{B_3} \sup_{x\in B_1} |x-y|^2 \,|K(x,y)|\,dy<+\infty
.\end{equation}

\begin{lemma}\label{lemCSv} Let~$\vartheta\in[0,2]$.
Let~$K\in\mathcal{K}_{0,\vartheta}^+$ satisfying~\eqref{CONZTINY}, \eqref{locint-RINF-PRE} and~\eqref{locint-RINF}. 
For every~$k\in\N$, let~$u_k \in \mathcal{V}_{K}$ and~$f_k$ be bounded and continuous
in $B_1$. Assume that
\begin{equation}\label{CONUK0v}
A u_k=f_k
\end{equation}
in~$B_1$ in the viscosity sense, that
\begin{equation*}
{\mbox{$f_k$ converges uniformly in~$B_1$ to some function~$f$ as~$k\to+\infty$,}}
\end{equation*}
that
\begin{equation}\label{CONUKv}
{\mbox{$u_k$ converges uniformly in~$B_4$ to some function~$u\in \mathcal{V}_{K} $ as~$k\to+\infty$}}
\end{equation}
and that
\begin{equation}\label{CONUK2v}
\lim_{k\to+\infty} \int_{\mathbb{R}^n\setminus B_3}\big|u(y)-u_k(y)
\big| \sup_{x\in B_1}|K(x,y)|\,dy=0.
\end{equation}
Then,
\begin{equation*}
A u=f
\end{equation*}
in~$B_1$ in the viscosity sense.
\end{lemma}

\begin{proof}
Let~$x_0\in B_1$ and $\rho>0$ such that~$B_\rho(x_0)\Subset B_1$.
Let~$\varphi\in C^2(\overline{B_\rho(x_0)})$ with~$\varphi=u$ outside~$B_\rho(x_0)$.
Suppose that~$v:=\varphi-u$ has a local maximum at~$x_0$.

We define, for every~$k\in\N$, 
$$ \varepsilon_k:=\| u-u_k\|_{L^\infty(B_1)}+\frac1k $$
and
$$ \varphi_k(x):=\begin{cases} \varphi(x)-\sqrt{\varepsilon_k}\,|x-x_0|^2 &{\mbox{ in }}\overline{B_\rho(x_0)},
\\
u_k(x) &{\mbox{ in }} \R^n\setminus \overline{B_\rho(x_0)}.
\end{cases}$$
We let~$v_k:=\varphi_k-u_k$ and~$x_k\in\overline{B_\rho(x_0)}$ be such that
$$v_k(x_k)=\max_{\overline{B_\rho(x_0)}}v_k.$$ We observe that
\begin{eqnarray*}
|x_k-x_0|^2&=&\frac{ \varphi(x_k)- \varphi_k(x_k)}{\sqrt{\varepsilon_k}}\\
&=&\frac{ v(x_k)- v_k(x_k)+u(x_k)-u_k(x_k)}{\sqrt{\varepsilon_k}}\\
&\le&\frac{ v(x_0)- v_k(x_0)+\varepsilon_k}{\sqrt{\varepsilon_k}}\\
&=&\frac{\varphi(x_0)- \varphi_k(x_0)-u(x_0)+u_k(x_0)+\varepsilon_k}{\sqrt{\varepsilon_k}}\\
&=&\frac{-u(x_0)+u_k(x_0)+\varepsilon_k}{\sqrt{\varepsilon_k}}\\&\le&2\sqrt{\varepsilon_k}.
\end{eqnarray*}
Thus, since~$\varepsilon_k$ is infinitesimal due to~\eqref{CONUKv},
we have that~$x_k$ converges to~$x_0$ as~$k\to+\infty$ and, in particular,
the function~$v_k$ has an interior maximum at~$x_k$.
This and~\eqref{CONUK0v} give that
\begin{equation}\label{MDDMR} 0 \le A \varphi_k(x_k) -f_k(x_k)\le A \varphi_k(x_k)-f(x_0)+|f(x_0)-f(x_k)|+
\|f-f_k\|_{L^\infty(B_1)}.\end{equation}

Now we claim that
\begin{equation}\label{ST:001v}
\lim_{k\to+\infty}
\int_{B_\rho(x_0)}\big(\varphi_k(x_k)-\varphi_k(y)\big)K(x_k,y)\,dy
=\int_{B_\rho(x_0)}\big(\varphi(x_0)-\varphi(y)\big)K(x_0,y)\,dy
.\end{equation}
To this end, we first observe that
\begin{equation}\label{-ST:002v}\begin{split}
\varphi_k(x_k)-\varphi_k(y)\,
=\,& \varphi(x_k)-\varphi(y)+\sqrt{\varepsilon_k}|y-x_0|^2-\sqrt{\varepsilon_k}|x_k-x_0|^2\\
=\,& \varphi(x_k)-\varphi(y)+\sqrt{\varepsilon_k}(2x_0-x_k-y)\cdot(x_k-y)
.\end{split}\end{equation}
We define
\begin{equation}\label{FYDE} F(y):=\sup_{x\in B_1} |x-y|^2 \, |K(x,y)|\end{equation}
and we observe that~$F\in L^1(B_3)$, thanks to~\eqref{locint-RINF}.
Accordingly, by the absolute continuity of the Lebesgue integrals,
for all~$\varepsilon>0$ there exists~$\delta>0$ such that if the Lebesgue
measure of a set~$Z\subset B_3$ is less than~$\delta$, then
\begin{equation}\label{uh8uh7uyh6tgf4rd4r}
\int_Z F(y)\,dy\le \varepsilon.
\end{equation}

We recall~\eqref{assym} and we see that, for $k$ sufficiently large,
\begin{eqnarray*}
&& \int_{B_{\rho/2}(x_k)}\big(\varphi_k(x_k)-\varphi_k(y)\big)K(x_k,y)\,dy\\
&=&\frac12 \int_{B_{\rho/2}}\big(\varphi_k(x_k)-\varphi_k(x_k+z)\big)K(x_k,x_k+z)\,dz
+\frac12\int_{B_{\rho/2}}\big(\varphi_k(x_k)-\varphi_k(x_k-z)\big)K(x_k,x_k-z)\,dz\\
&=&\frac12 \int_{B_{\rho/2}}\big(2\varphi_k(x_k)-\varphi_k(x_k+z)-\varphi_k(x_k-z)\big)K(x_k,x_k+z)\,dz\\
&=&\frac12 \int_{B_\rho(x_0)}\chi_{B_{\rho/2}(x_k)}(y)\big(2\varphi_k(x_k)-\varphi_k(y)-\varphi_k(2x_k-y)\big)K(x_k,y)\,dy.
\end{eqnarray*}
Thus, for all~$y\in B_\rho(x_0)$ we define
$$ \zeta_k(y):=
\frac12 \chi_{B_{\rho/2}(x_k)}(y)\big(2\varphi_k(x_k)-\varphi_k(y)-\varphi_k(2x_k-y)\big)K(x_k,y)$$
and we point out that
\begin{eqnarray*} |\zeta_k(y)|&\le&\frac12
\Big|\big(\varphi_k(x_k)-\varphi_k(y)\big)+\big(\varphi_k(x_k)-\varphi_k(2x_k-y)\big)\Big|\,|K(x_k,y)|
\\&=&
\frac12
\left|\left( \int_0^1 \nabla\varphi_k\big(tx_k+(1-t)y\big)\,dt
-\int_0^1 \nabla\varphi_k\big(tx_k+(1-t)(2x_k-y)\big)\,dt \right)\cdot(x_k-y)\right|\,|K(x_k,y)|\\
&=&
\Bigg|
 \int_0^1\left((1-t)\int_0^1 D^2\varphi_k
 \Big(\tau\big(tx_k+(1-t)y\big)+(1-\tau)\big(tx_k+(1-t)(2x_k-y)\big)\Big) \,d\tau\right)\,dt
\\&&\qquad (x_k-y)
\cdot(x_k-y)\Bigg|\,|K(x_k,y)|\\&\le&C\,|x_k-y|^2\,|K(x_k,y)|\\&\le& C\,F(y),
\end{eqnarray*}
where the notation in~\eqref{FYDE} was used.
In particular, since~$F\in L^1(B_3)$ by~\eqref{locint-RINF},
we can exploit
the absolute continuity of the Lebesgue integrals (see~\eqref{uh8uh7uyh6tgf4rd4r}) and deduce that
for all~$\varepsilon>0$ there exists~$\delta>0$ such that if the Lebesgue
measure of a set~$Z\subset B_\rho(x_0)$ is less than~$\delta$, then
$$ \int_Z|\zeta_k(y)|\,dy\le C\varepsilon.$$
Hence, recalling~\eqref{CONZTINY},
we utilize the Vitali Convergence Theorem
and obtain that
\begin{eqnarray*}&& \lim_{k\to+\infty}
\int_{B_\rho(x_0)}\zeta_k(y)\,dy=
\int_{B_\rho(x_0)}\lim_{k\to+\infty}\zeta_k(y)\,dy\\&&\qquad=\frac12
\int_{B_{\rho/2}(x_0)}
\big(2\varphi(x_0)-\varphi(y)-\varphi(2x_0-y)\big)K(x_0,y)
\end{eqnarray*}
which proves that
\begin{equation}\label{ST:001-BMEZZI}
\lim_{k\to+\infty}
\int_{B_{\rho/2}(x_0)}\big(\varphi_k(x_k)-\varphi_k(y)\big)K(x_k,y)\,dy
=\int_{B_{\rho/2}(x_0)}\big(\varphi(x_0)-\varphi(y)\big)K(x_0,y)\,dy
.\end{equation}
Now, for every~$y\in \R^n\setminus B_{\rho/2}(x_0)$ we define
\begin{equation*} \eta_k(y):=\big(\varphi_k(x_k)-\varphi_k(y)\big)K(x_k,y).\end{equation*}
We observe that, if~$k$ is sufficiently large,
for every~$y\in B_3\setminus B_{\rho/2}(x_0)$,
\begin{equation*} |\eta_k(y)|\le C\,\sup_{x\in B_3\setminus B_{\rho/4}(y)}|K(x,y)| ,\end{equation*}
which belongs to~$L^1(B_\rho(x_0)\setminus B_{\rho/2}(x_0))$ due to~\eqref{locint-RINF-PRE}.

Consequently, by~\eqref{CONZTINY} and the Dominated Convergence Theorem,
\begin{equation*}\begin{split}&
\lim_{k\to+\infty}
\int_{B_\rho(x_0)\setminus B_{\rho/2}(x_0)}\big(\varphi_k(x_k)-\varphi_k(y)\big)K(x_k,y)\,dy=
\lim_{k\to+\infty}
\int_{B_\rho(x_0)\setminus B_{\rho/2}(x_0)}\eta_k(y)\,dy\\&\qquad
=
\int_{B_\rho(x_0)\setminus B_{\rho/2}(x_0)}\lim_{k\to+\infty}
\eta_k(y)\,dy
=\int_{B_\rho(x_0)\setminus B_{\rho/2}(x_0)}\big(\varphi(x_0)-\varphi(y)\big)K(x_0,y)\,dy
.\end{split}\end{equation*}
This and~\eqref{ST:001-BMEZZI}
show the validity of~\eqref{ST:001v}.

Having completed the proof of~\eqref{ST:001v}, we now claim that
\begin{equation}\label{ST:002}
\lim_{k\to+\infty}
\int_{\R^n\setminus B_\rho(x_0)}\big(\varphi_k(x_k)-\varphi_k(y)\big)K(x_k,y)\,dy
=\int_{\R^n\setminus B_\rho(x_0)}\big(\varphi(x_0)-\varphi(y)\big)K(x_0,y)\,dy
.\end{equation}

Indeed, for all~$y\in\R^n\setminus \overline{B_\rho(x_0)}$ we set
$$ \mu_k(y):=\big(\varphi_k(x_k)-\varphi_k(y)\big)K(x_k,y)=
\big(\varphi(x_k)-\sqrt{\varepsilon_k}|x_k-x_0|^2-u_k(y)\big)K(x_k,y).$$
If~$y\in B_3\setminus \overline{B_\rho(x_0)}$
and~$k$ is sufficiently large, then, by~\eqref{CONUKv},
\begin{eqnarray*}
|\mu_k(y)|\le C\,(1+|u(y)|) \,|K(x_k,y)|\le
C\,(1+\|u\|_{L^\infty(B_3)}) \,\sup_{x\in B_3\setminus B_{\rho/2}(y)}
|K(x,y)|
\end{eqnarray*}
and the latter function belongs to~$L^1(B_3\setminus \overline{B_\rho(x_0)})$
owing to~\eqref{locint-RINF-PRE}.

This, \eqref{CONZTINY}, \eqref{CONUKv} and the Dominated Convergence Theorem
lead to
\begin{equation}\label{TGBDIUNDG9TGRSEI-O}
\begin{split}&
\lim_{k\to+\infty}\int_{B_3\setminus{B_\rho(x_0)}}\big(\varphi_k(x_k)-\varphi_k(y)\big)K(x_k,y)\,dy\,\\&\qquad=\,
\int_{B_3\setminus{B_\rho(x_0)}}\lim_{k\to+\infty}
\big(\varphi(x_k)-\sqrt{\varepsilon_k}|x_k-x_0|^2-u_k(y)\big)K(x_k,y)
\,dy\\&\qquad=\,\int_{B_3\setminus{B_\rho(x_0)}}
\big(\varphi(x_0)-u(y)\big)K(x_0,y)
\,dy.
\end{split}\end{equation}

In light of~\eqref{mcondv} and~\eqref{CONUK2v} we also remark that, for large~$k$,
\begin{equation}\label{ILBOPUN}
\begin{split}&
\int_{\R^n\setminus B_3} |u_k(y)|\sup_{{x\in B_1}}
K(x,y)\,dy\\&\qquad\le \int_{\mathbb{R}^n\setminus B_3}\big|u(y)-u_k(y)
\big| \sup_{x\in B_1}K(x,y) \,dy
+\int_{\R^n\setminus B_3} |u(y)|\sup_{{x\in B_1}}
K(x,y)\,dy\\&\qquad\le1+
\int_{\R^n\setminus B_3} |u(y)|\sup_{{x\in B_1}}
K(x,y)\,dy\le C,\end{split}\end{equation}
up to renaming~$C$ once again.
Furthermore, if~$y\in\R^n\setminus B_3$ and~$\delta_k:=\big|\varphi(x_k)-\varphi(x_0)\big|+\sqrt{\varepsilon_k}$,
we have that~$\delta_k$ is infinitesimal as~$k\to+\infty$ and
\begin{eqnarray*}&&
\Big| \big(\varphi_k(x_k)-\varphi_k(y)\big)K(x_k,y)
-\big(\varphi(x_0)-u(y)\big)
K(x_0,y)\Big|\\&\le& 
\Big(
\big|\varphi(x_k)-\varphi(x_0)\big|+\sqrt{\varepsilon_k}+|u_k(y)-u(y)|\Big)
|K(x_k,y)|
+\big|\varphi(x_0)-u(y)\big|\,\big|K(x_0,y)-K(x_k,y)\big|\\
&\le&\delta_k \sup_{x\in B_1}|K(x,y)| +|u_k(y)-u(y)|\sup_{x\in B_1}
|K(x,y)|+
C\,(1+|u(y)|)\,\big|K(x_0,y)-K(x_k,y)\big|\\
&\le&\delta_k \sup_{x\in B_1\setminus B_1(y)}|K(x,y)|
+|u_k(y)-u(y)|\sup_{x\in B_1}|K(x,y)|+
C\,(1+|u(y)|)\,\big|K(x_0,y)-K(x_k,y)\big|.
\end{eqnarray*}
Gathering this information, \eqref{locint-RINF-PRE} and~\eqref{CONUK2v}, we find that
\begin{equation}\label{DIUIPN7ujGyjhmEAFS}\begin{split}&
\lim_{k\to+\infty}
\int_{\R^n\setminus B_3}\Big| \big(\varphi_k(x_k)-\varphi_k(y)\big)K(x_k,y)-\big(\varphi(x_0)-u(y)\big)K(x_0,y)\Big|\,dy\\&\qquad\le
\lim_{k\to+\infty}
C\int_{\R^n\setminus B_3}(1+|u(y)|)\,\big|K(x_0,y)-K(x_k,y)\big|\,dy.
\end{split}
\end{equation}
We also point out that, if~$y\in\R^n\setminus B_3$,
$$ (1+|u(y)|)\,\big|K(x_0,y)-K(x_k,y)\big|\le
2(1+|u(y)|)\,\sup_{x\in B_1\setminus B_1(y)} |K(x,y)|$$
and the latter function belongs to~$L^1(\R^n\setminus B_3)$, thanks to~\eqref{mcondv}
and~\eqref{locint-RINF-PRE}. The Dominated Convergence Theorem and~\eqref{CONZTINY}
thereby give that
$$ \lim_{k\to+\infty}
\int_{\R^n\setminus B_3}(1+|u(y)|)\,\big|K(x_0,y)-K(x_k,y)\big|\,dy
=0.$$
This and~\eqref{DIUIPN7ujGyjhmEAFS} yield that
$$ \lim_{k\to+\infty}
\int_{\R^n\setminus B_3}\Big| \big(\varphi_k(x_k)-\varphi_k(y)\big)K(x_k,y)-\big(\varphi(x_0)-u(y)\big)K(x_0,y)\Big|\,dy=0,$$
which, combined with~\eqref{TGBDIUNDG9TGRSEI-O}, proves~\eqref{ST:002}.

By combining~\eqref{ST:001v} and~\eqref{ST:002}, we deduce that~$A \varphi_k(x_k)\to A \varphi(x_0)$ as~$k\to+\infty$.
As a result, by passing to the limit in~\eqref{MDDMR}, we conclude that~$A \varphi (x_0)\ge f(x_0)$.

Similarly, one sees that if the function~$\varphi-u$ has a local minimum at~$x_0$ then~$A \varphi (x_0)\le f(x_0)$.
\end{proof}

\begin{corollary} \label{corequivv} Let~$\vartheta\in[0,2]$.
Let~$K\in\mathcal{K}_{0,\vartheta}^+$ 
satisfying~\eqref{CONZTINY}, \eqref{locint-RINF-PRE} and~\eqref{locint-RINF}
and~$u \in \mathcal{V}_{K}$. Let $f$ be bounded and continuous in $B_1$.

Then
$$Au=f {\mbox{ in }}B_1$$
in viscosity sense is equivalent to
$$Au\stackrel{0}{=}f {\mbox{ in }}B_1$$
in the viscosity sense of Definition~\ref{defv}.
\end{corollary}

\begin{proof} Suppose first that~$Au=f$ in~$B_1$
in viscosity sense. For every~$R>5$, we define
$$f_R(x):=f(x)+\int_{B_R^c}u(y)K(x,y)\,dy.$$
Notice that
$$ \sup_{x\in B_1} |f_R(x)-f(x)|
\le\int_{B_R^c}|u(y)|\sup_{x\in B_1} K(x,y)\,dy
,$$
which is infinitesimal, thanks to~\eqref{mcondv} (used here with~$m:=0$),
and thus~$f_R$ converges to~$f$ uniformly in~$B_1$.

Our objective is to prove that~$A(\chi_R u)=f_R$ in~$B_1$ in the viscosity sense
(from which we obtain that~$Au\stackrel{0}{=}f$ in~$B_1$
in the viscosity sense of Definition~\ref{defv}).

To check this claim, we pick a point~$x_0\in B_1$ and touch~$\chi_R u$
from below by a test function~$\varphi$ at~$x_0$, with~$\varphi=\chi_R u$
outside~$B_2$.
We define~$\psi:=\varphi+(1-\chi_R)u$ and we observe that~$\psi$
touches~$u$ by below at~$x_0$ and that~$\psi=u$ outside~$B_2$.
As a result, $A\psi(x_0)\ge f(x_0)$ and therefore
\begin{eqnarray*}
f_R(x_0)&=& f(x_0)+\int_{B_R^c}u(y)K(x_0,y)\,dy\\
&\le&A\psi(x_0)+\int_{B_R^c}u(y)K(x_0,y)\,dy\\
&=&\int_{\R^n}(\psi(x_0)-\psi(y))K(x_0,y)\,dy
+\int_{B_R^c}u(y)K(x_0,y)\,dy\\
&=&\int_{\R^n}(\varphi(x_0)-\psi(y))K(x_0,y)\,dy
+\int_{B_R^c}u(y)K(x_0,y)\,dy\\
&=&\int_{B_R}(\varphi(x_0)-\varphi(y))K(x_0,y)\,dy+
\int_{B_R^c}(\varphi(x_0)-\psi(y))K(x_0,y)\,dy
+\int_{B_R^c}u(y)K(x_0,y)\,dy\\&=&A\varphi(x_0)+
\int_{B_R^c}(\varphi(y)-\psi(y))K(x_0,y)\,dy
+\int_{B_R^c}u(y)K(x_0,y)\,dy\\&=&
A\varphi(x_0).
\end{eqnarray*}
Similarly, if~$\varphi$ touches~$\chi_R u$
from above, then~$A \varphi(x_0)\le f_R(x_0)$.
These observations entail that~$A(\chi_R u)=f_R$ in~$B_1$ in the viscosity sense.

This proves one of the implications of Corollary~\ref{corequivv}.
To prove the other, we assume now that~$Au\stackrel{0}{=}f$ in~$B_1$
in the viscosity sense of Definition~\ref{defv}.
Then, we find~$f_R: B_1 \to \R$ such that~$
A(\chi_R u)=f_R$
in $B_1$ in viscosity sense, with~$f_R$ converging
to~$f$ uniformly in~$B_1$.
We remark that
\begin{eqnarray*}
\lim_{R\to+\infty} \int_{\mathbb{R}^n\setminus B_3}\big|u(y)-(\chi_R u)(y)
\big| \sup_{x\in B_1}|K(x,y)|\,dy=0,
\end{eqnarray*}
thanks to~\eqref{mcondv} (used here with~$m:=0$) and the Dominated Convergence Theorem.

We can therefore apply Lemma~\ref{lemCSv} and conclude that~$Au=f$
in~$B_1$ in the sense of viscosity, as desired.
\end{proof}

In the next result we state the viscosity counterpart
of Lemma~\ref{lemmaj} (its proof is omitted since
it is similar to the one of Lemma~\ref{lemmaj}, just noticing that
the functions~$v$ and~$\chi_R w-\chi_4 w$ vanish in~$B_4$,
hence the viscous and pointwise setting would equally apply to them).

\begin{lemma} \label{lemmajv} 
Let $j, m \in \N_0$, with~$j\le m$, $\vartheta\in[0,2]$
and $K \in \mathcal{K}_{j,\vartheta}^+\cap \mathcal{K}_{m,\vartheta}^+$.
Let~$f$ be bounded and continuous in~$B_1$ and let~$u \in \mathcal{V}_{K}$
such that
\begin{equation}\label{d3ivbrgretgyt57tv}
Au\stackrel{m}{=}f\end{equation} in $B_1$. 

Then, there exist a function $\bar f$ and a polynomial $P$ of degree at most $m-1$
such that~$\bar f=f+P$ and~$Au\stackrel{j}{=}\bar f$ in $B_1$.
\end{lemma}

Next, as a possible application, we show a
specific case in which the existence of solution to a Dirichlet problem is guaranteed. For this, we consider a
family of kernels comparable to the fractional Laplace operator, as follows.
For any $s \in (0,1)$, given real numbers $\Lambda\ge \lambda >0$, we consider the family of
kernels~$K$ as defined in~\eqref{FL}. We suppose that
\begin{equation}\label{emmecond}
{\mbox{there exists~$m\in\N_0$ such that condition~\eqref{taylor} is satisfied.}}
\end{equation}
We also assume that condition~\eqref{assym} holds true
and that~$K$ is
translation invariant, i.e. 
\begin{equation}\label{4.30BIS}
{\mbox{$K(x+z,y+z)=K(x,y)$ for any $x$, $y$, $z \in \R^n$.}}\end{equation}
With this, we have that~$K$ belongs to $\mathcal{K}_{m,\vartheta}^+$, with~$m$ as in~\eqref{emmecond}
and for every~$\vartheta\in(2s,2]$. Moreover, it also satisfies~\eqref{locint-RINF-PRE} and~\eqref{locint-RINF}.

We introduce the fractional Sobolev space
$$\mathbb{H} ^s(B_1):=\left\{ u \in L^2(B_1): \iint_{\R^{2n}\setminus (B_1^c\times B_1^c)} (u(x)-u(y))^2 
K(x,y)\,dy <+\infty \right\}$$
and, given~$g\in \mathbb{H} ^s(B_1)$,  the class 
$$J_g(B_1):=\Big\{ u \in \mathbb{H}^s(B_1):\; u=g \; {\mbox{ in }} B_1^c\Big\}.$$
We use this class to seek solutions to the Dirichlet problem (see~\cite{P18}).
More precisely, the following result can be proved by using the Direct Methods of the Calculus of Variations and
the strict convexity of the functional.

\begin{proposition} \label{pala}
Let $K$ be as in \eqref{FL}, 
\eqref{assym},
\eqref{emmecond}
and~\eqref{4.30BIS}.
Let~$f\in L^2(B_1)$ and $g \in \mathbb{H}^s(B_1)$.
Then, there exists a unique minimizer of the functional
\begin{equation} \label{energy}
\mathcal{E}(u):=\frac{1}{4}\iint_{\R^n\times\R^n} (u(x)-u(y))^2 K(x,y)\,dx\,dy-\int_{B_1}f(x)u(x)\,dx
\end{equation}
over $J_g(B_1)$.

In addition, $u \in J_g(B_1)$ is a minimizer of \eqref{energy} over $J_g(B_1)$ if and only if it is a weak solution of 
\begin{equation} \label{weak}
\begin{cases}
Au=f& {\mbox{ in }}B_1,\\
u=g& \mbox{ in } B_1^c,
\end{cases}
\end{equation}
that is, for every~$\phi\in C^\infty_0(B_1)$,
$$ \frac12\iint_{\R^n\times\R^n}\big(u(x)-u(y)\big)\big(\phi(x)-\phi(y)\big)\,K(x,y)\,dx\,dy=\int_{B_1} f(x)\phi(x)\,dx.$$
\end{proposition}

The next result is a generalization of Theorem 2 in \cite{SV14}, which shows the global continuity of 
weak solutions of an equation which includes the operator of our interest.

\begin{proposition} \label{contsol}
Let $K$ be as in~\eqref{FL}
\eqref{assym},
\eqref{emmecond}
and~\eqref{4.30BIS}. Let $f \in L^\infty(B_1)$ and~$g\in C^\alpha(\R^n)$
for some~$\alpha\in(0,\min\{2s,1\})$. Assume that
$$ |g(x)|\le C|x|^\alpha\qquad{\mbox{ for all }}x\in\R^n\setminus B_1.$$
Let also $u \in J_g(B_1)$ be a weak solution of 
\begin{equation} \label{weak2}
Au=f\quad {\mbox{ in }}B_1.
\end{equation}
Then, $u \in C(\R^n)$.
\end{proposition}
\begin{proof}
First of all, we exploit
Proposition~\ref{pala} with~$g:=0$ to find a weak solution~$v$ of
$$ \begin{cases}
Av=f & {\mbox{ in }}B_1,\\
v=0& {\mbox{ in }}B_1^c.
\end{cases}$$
By Proposition~7.2 in~\cite{R16}
(see also~\cite{MR3293447} for related results), we have that~$v\in C(\R^n)$.

Let now~$w:=u-v$. We see that~$w$ is a weak solution of
$$ \begin{cases}
Aw=0 & {\mbox{ in }}B_1,\\
w=u=g& {\mbox{ in }}B_1^c.
\end{cases}$$
We thus exploit Theorem~1.4 in~\cite{MR4135310} and find that~$w\in C(\R^n)$.
{F}rom these observations, we find that~$u=v+w\in C(\R^n)$, as desired.
\end{proof}

With this, we can now prove that, in this setting, week solutions are also viscosity solutions.
 
\begin{proposition} \label{wisv}
Let $K$ be as in~\eqref{FL}
\eqref{assym}, \eqref{CONZTINY},
\eqref{emmecond}
and~\eqref{4.30BIS}. Let $f$ be bounded and continuous in~$B_1$ and~$g\in C^\alpha(\R^n)$
for some~$\alpha\in(0,\min\{2s,1\})$. Assume that
\begin{equation}\label{JS-CREg} |g(x)|\le C|x|^\alpha\qquad{\mbox{ for all }}x\in\R^n\setminus B_1.\end{equation}
Let also $u \in J_g(B_1)$ be a weak solution of 
\begin{equation} \label{weakd}
\begin{cases} 
Au = f&{\mbox{ in }} B_1,\\
u=g&{\mbox{ in }} B_1^c.
\end{cases}
\end{equation}
Then, $u$ is a viscosity solution of \eqref{weakd}.
\end{proposition}

\begin{proof}
By Proposition~\ref{contsol}, we know that~$u \in C(\R^n)$. 

Now, we take a point $x_0 \in B_1$
and a function~$\rho\in C^\infty_0( B_1,[0,1])$ and we
consider an even mollifier~$\rho_\varepsilon:=\varepsilon^{-n}\rho(x/\varepsilon)$, for any~$\varepsilon\in(0,1)$. 
We set~$u_\varepsilon:=u \ast \rho_\varepsilon$ and~$
f_\varepsilon:= f \ast \rho_\varepsilon$ (where we identified~$f$ with its null extension outside~$B_1$).

We claim that
\begin{equation}\label{equamol}
{\mbox{$Au_\varepsilon =f_\varepsilon$ in the weak sense in any ball~$B_r(x_0)$ such that~$B_r(x_0)\Subset B_1$.}}
\end{equation}
To prove this, we take a ball~$B_r(x_0)$ such that~$B_r(x_0)\Subset B_1$ and
a function~$\varphi \in C^\infty_0(B_\rho (x_0))$. We observe that
\begin{equation} \label{triplesp}
\begin{split}&
\int_{\R^n}\left(\;\iint_{\R^{2n}}(u(x+z)-u(y+z))(\varphi(x)-\varphi(y))\rho_\varepsilon(z)K(x,y)\,dx\,dy\right)\,dz \\
=\,&\int_{B_\varepsilon}\left(\quad \iint_{\R^{2n}\setminus (B_r^c(x_0)\times B^c_r(x_0))} (u(x+z)-u(y+z))(\varphi(x)-\varphi(y))\rho_\varepsilon(z)K(x,y)\,dx\,dy \right)\,dz \\
\le\,& \frac{\varepsilon^{-n}}{2}\int_{B_\varepsilon} \left(\quad \iint_{\R^{2n}\setminus (B_1^c\times B^c_1)} (u(x)-u(y))^2 K(x,y)\,dx\,dy +\iint_{\R^{2n}\setminus (B_1^c\times B^c_1)} (\varphi(x)-\varphi(y))^2 K(x,y)\,dx\,dy \right)\,dz \\
<\,& +\infty,
\end{split}
\end{equation}
thanks to~\eqref{4.30BIS}.

Therefore, Tonelli's Theorem gives us that the function
$$(x,y,z)\in \R^{2n}\times \R^n \mapsto
(u(x+z)-u(y+z))(\varphi(x)-\varphi(y))\rho_\varepsilon(z) K(x,y)\quad {\mbox{lies in }} L^1(\R^{2n}\times \R^n).$$
One can interchange the order of integration in \eqref{triplesp}, thanks to Fubini's Theorem, and exploit the definition of $u_\varepsilon$ to obtain
\begin{equation}
\begin{split}
\int_{\R^n}&\left(\;\iint_{\R^{2n}}(u(x+z)-u(y+z))(\varphi(x)-\varphi(y))\rho_\varepsilon(z)K(x,y)\,dx\,dy\right)\,dz \\
&=\iint_{\R^{2n}}\left(\;\int_{\R^{n}}(u(x+z)-u(y+z))(\varphi(x)-\varphi(y))\rho_\varepsilon(z)K(x,y)\,dz\right)\,dx\,dy \\
&=\iint_{\R^{2n}}(u_\varepsilon(x)-u_\varepsilon(y))(\varphi(x)-\varphi(y))K(x,y)\,dx\,dy.
\end{split}
\end{equation}
Then, we can use Fubini's Theorem once again to get
\begin{equation}
\begin{split}
\int_{\R^n} & f_\varepsilon (x)\varphi(x)\,dx \\
&=\int_{\R^n}\left(\;\int_{\R^n}f(x+z)\rho_\varepsilon(z)\varphi(x)\,dx \right)\,dz \\
&=\int_{B_\varepsilon}\left(\;\int_{\R^n}f(\tilde x)\varphi(\tilde x-z)\rho_\varepsilon(z)d\tilde x \right)
\,dz \\
&=\frac12\int_{B_\varepsilon}\left(\;\iint_{\R^{2n}}(u(\tilde x)-u(\tilde y))(\varphi(\tilde x-z)-\varphi(\tilde y-z))\rho_\varepsilon(z)K(\tilde x,\tilde y)d\tilde xd\tilde y\right)\,dz \\
&=\frac12\int_{\R^n}\left(\;\iint_{\R^{2n}}(u(x+z)-u(y+z))(\varphi(x)-\varphi(y))\rho_\varepsilon(z)K(x,y)\,dx\,dy\right)\,dz \\
&=\frac12\iint_{\R^{2n}}(u_\varepsilon(x)-u_\varepsilon(y))(\varphi(x)-\varphi(y))K(x,y)\,dx\,dy,
\end{split}
\end{equation}
since the kernel~$K$
is translation invariant and $u$ satisfies~\eqref{weakd} in weak sense. This shows~\eqref{equamol}.

Now, given~$B_r(x_0)\Subset B_1$, we show that
\begin{equation}\label{123SUOHBVDGVB9yuhfed}
{\mbox{the map 
}}B_r(x_0)\ni x\mapsto \int_{\R^{n}} (u_\varepsilon(x)-u_\varepsilon(y))\,K(x,y)\,dy{\mbox{
is continuous.}}
\end{equation}
For this, we let~$x_k$ be a sequence converging to a given point~$x\in B_r(x_0)$ and we define
$$\zeta_k(z):=(2u_\varepsilon(x_k)-u_\varepsilon(x_k+z)-u_\varepsilon(x_k-z))\,K(0,z).$$
Since~$u_\varepsilon$ is smooth and its growth at infinity is controlled via~\eqref{JS-CREg}, we know that
$$ |2u_\varepsilon(x_k)-u_\varepsilon(x_k+z)-u_\varepsilon(x_k-z)|\le C_\e \min\{ |z|^2, |z|^\alpha\},$$
for some~$C_\e>0$. For this reason and~\eqref{FL},
$$ |\zeta_k(z)|\le\frac{C_\e \min\{ |z|^2, |z|^\alpha\}}{|z|^{n+2s}},$$
up to renaming~$C_\e$ and therefore we are in the position of applying the Dominated Convergence Theorem and
conclude that
$$ \lim_{k\to+\infty}
\int_{\R^n}(2u_\varepsilon(x_k)-u_\varepsilon(x_k+z)-u_\varepsilon(x_k-z))\,K(0,z)\,dz
=
\int_{\R^n}(2u_\varepsilon(x)-u_\varepsilon(x+z)-u_\varepsilon(x-z))\,K(0,z)\,dz.$$
In view of~\eqref{assym} and~\eqref{4.30BIS},
this proves~\eqref{123SUOHBVDGVB9yuhfed}.

We also observe that
\begin{equation}\label{punt0438469}
{\mbox{$Au_\varepsilon =f_\varepsilon$ pointwise in any ball~$B_r(x_0)$ such that~$B_r(x_0)\Subset B_1$.}}
\end{equation}
Indeed, by~\eqref{assym},
\eqref{equamol} and~\eqref{123SUOHBVDGVB9yuhfed},
if~$x\in B_r(x_0)$ and~$\varphi\in C^\infty_0(B_r(x_0))$,
\begin{eqnarray*}
\int_{\R^n} f_\e(x)\varphi(x)\,dx=
\iint_{\R^{2n}}(u_\varepsilon(x)-u_\varepsilon(y)) \varphi(x)K(x,y)\,dx\,dy.
\end{eqnarray*}
Since~$\varphi$ is arbitrary, we arrive at
\begin{eqnarray*}
f_\e(x)=\int_{\R^{n}}(u_\varepsilon(x)-u_\varepsilon(y)) K(x,y)\,dy,
\end{eqnarray*}
from which we obtain~\eqref{punt0438469}.

We also have that
\begin{equation}\label{visc0punt0438469}
{\mbox{$Au_\varepsilon =f_\varepsilon$ in the viscosity sense in any ball~$B_r(x_0)$ such that~$B_r(x_0)\Subset B_1$.}}
\end{equation}
For this, we take a smooth function~$\psi$ touching, say from below, the function~$u_\e$ at some point~$p\in
B_r(x_0)$. Since the kernel~$K$ is positive (thanks to~\eqref{FL}) and recalling~\eqref{punt0438469}, we have that
\begin{eqnarray*}
&& f_\e(p)=\int_{\R^{n}}(u_\varepsilon(p)-u_\varepsilon(y)) K(p,y)\,dy
=\int_{\R^{n}}(\psi(p)-u_\varepsilon(y)) K(p,y)\,dy\\&&\qquad\qquad
\le\int_{\R^{n}}(\psi(p)-\psi(y)) K(p,y)\,dy=A\psi(p).
\end{eqnarray*}
This and a similar computation when~$\psi$ touches from above give~\eqref{visc0punt0438469}.

We also remark that~$u_\e$ and~$f_\e$ converge uniformly to~$u$ and~$f$, respectively,
in any ball~$B_r(x_0)\Subset B_1$, due to Theorem~9.8 in~\cite{MR3381284}.
In addition, by~\eqref{JS-CREg}, we see that, for every~$y\in\R^n\setminus B_{3r}(x_0)$,
$$ |u_\e(y)|\le \int_{B_\e} |u(y-z)|\rho_\e(z)\,dz\le C \int_{B_\e} |y-z|^\alpha \rho_\e(z)\,dz
\le C|y|^\alpha,
$$
up to renaming~$C>0$. As a consequence of this and~\eqref{FL}, we have that, for every~$y\in\R^n\setminus B _{3r}(x_0)$,
$$ |u(y)-u_\e(y)|\,\sup_{x\in B_r(x_0)}
 K(x,y)\le C |y|^\alpha \sup_{x\in B_r(x_0)}\frac{1}{|x-y|^{n+2s}}\le \frac{C}{|y|^{n+2s-\alpha}},$$
up to relabeling~$C>0$. Since~$\alpha<\min\{2s,1\}$, this function is in $L^1(\R^n\setminus  B_{3r}(x_0))$, and therefore
we exploit the Dominated Convergence Theorem to obtain that
$$ \lim_{\e\searrow0} \int_{\R^n\setminus B_{3r}(x_0)}|u(y)-u_\e(y)|\,\sup_{x\in B_1} K(x,y)\,dy=0.$$
Consequently, condition~\eqref{CONUK2v} is satisfied, and
therefore we can apply Lemma~\ref{lemCSv}, thus obtaining that~$Au=f$
in the viscosity sense, as desired.
\end{proof}

With this preliminary work, we can now address the existence of solutions for
a Dirichlet problem in a generalized setting. 

\begin{theorem}\label{EQUI90-97900BVA}
Let $K$ be as in~\eqref{FL}
\eqref{assym}, \eqref{CONZTINY},
\eqref{emmecond}
and~\eqref{4.30BIS}. Let $f$ be bounded and continuous in~$B_1$ and~$g\in C^\alpha(\R^n)$
for some~$\alpha\in(0,\min\{2s,1\})$. Assume that
$$ |g(x)|\le C|x|^\alpha\qquad{\mbox{ for all }}x\in\R^n\setminus B_1.$$

Then, there exists a
function $u \in \mathcal{V}_{K}$ such that
\begin{equation}  \label{DFD}
\begin{cases} 
Au\stackrel{m}{=}f&{\mbox{ in }} B_1,\\
u=g& {\mbox{ in }}B_1^c.
\end{cases}
\end{equation}
Also, the solution to~\eqref{DFD} is not unique, since
the space of solutions of \eqref{DFD} has dimension $N_m$, with
\begin{equation*} \label{}
N_m:=
\sum_{j=0}^{m-1}
\binom{j+n-1}{n-1}.
\end{equation*}
\end{theorem}

\begin{proof}
Firstly, we prove the existence of solutions for~\eqref{DFD}. For this goal, we set 
\begin{equation*}
u_1:=\chi_{B_4^c}\,g \qquad {\mbox{and }} \qquad\tilde g:=\chi_{B_4\setminus B_1} \,g.
\end{equation*}
Since $u_1$ is identically zero in $B_4$ and
$K\in \mathcal{K}_{m,\vartheta}^+$, we can write $Au_1\stackrel{m}{=}f_{u_1}$ in $B_1$ in both
pointwise and viscosity sense,
for some function $f_{u_1}$, due to Remark~\ref{ifsmooth} and Lemma~\ref{EQUIBVA}.

We now define $\tilde f:=f-f_{u_1}$ and consider
the Dirichlet problem given by 
\begin{equation} \label{VDP}
\begin{cases} 
A\tilde u=\tilde f&{\mbox{ in }} B_1,\\
\tilde u= \tilde g&{\mbox{ in }} B_1^c.
\end{cases}
\end{equation}
By Proposition \ref{pala}, we find that \eqref{VDP} has a unique weak solution $\tilde u$.
Moreover, thanks to Proposition~\ref{wisv}, we get that $\tilde u$ is a viscosity solution of~\eqref{VDP}.

Furthermore, by Remark \ref{higherorderv} and Corollary~\ref{corequivv}
we obtain that
\begin{equation*} 
\begin{cases} 
A\tilde u\stackrel{m}{=}\tilde f& {\mbox{ in }}B_1,\\
\tilde u= \tilde g&{\mbox{ in }}B_1^c.
\end{cases}
\end{equation*}
Now, we set $u:=u_1+\tilde u$ and we get that $Au=Au_1+A
\tilde u\stackrel{m}{=}f_{u_1}+\tilde f=f$ in $B_1$.
Moreover, we have that~$u=u_1+\tilde g=g$ in $B_1^c$. These observations give that is~$u$ is solution of \eqref{DFD}.
This proves the existence of solution for~\eqref{DFD}.

Now, we focus on the second part of the proof. Namely we establish that solutions of~\eqref{DFD}
are not unique and we determine the dimension
of the corresponding linear space. For this, we notice that, exploiting Propositions~\ref{pala}
and~\ref{wisv}, one can find a unique solution $\tilde u_P$ of the problem
\begin{equation}\label{PIUUETv}
\begin{cases} 
A\tilde u_P=P&{\mbox{ in }}B_1,\\
\tilde u_P=0&{\mbox{ in }} B_1^c.
\end{cases}
\end{equation}
Furthermore, $A\tilde u_P\stackrel{0}{=}P$ in $B_1$, due to Corollary \ref{corequivv}.
Using Remark~\ref{higherorderv}, we obtain that $A\tilde u_P\stackrel{m}{=}P$ in $B_1$.
Moreover, from Remark~\ref{plusPv}, we obtain that $\tilde u_P$ is a solution of
\begin{equation} \label{Pvisc} 
\begin{cases} 
A\tilde u_P\stackrel{m}{=}0&{\mbox{ in }}B_1,\\
\tilde u_P=0&{\mbox{ in }}B_1^c.
\end{cases}
\end{equation}
This yields that if $u$ is a solution of \eqref{DFD}, then $u+\tilde u_P$ is also a solution of \eqref{DFD}.

Viceversa, if $u$ and $v$ are two solutions of~\eqref{DFD}, then $w:=u-v$ is a solution of
\begin{equation*} 
\begin{cases} 
Aw\stackrel{m}{=}0&{\mbox{ in }}B_1,\\
w=0&{\mbox{ in }}B_1^c.
\end{cases}
\end{equation*}
Here we can apply Lemma \ref{lemmajv} with $j:=0$ thus obtaining that $Aw\stackrel{0}{=}P$ in $B_1$, where $P$ is a polynomial of $\deg P \le m-1$. We use Corollary \ref{corequivv} one more time to find that
\begin{equation}\label{KSMD-DIVv} 
\begin{cases} 
Aw=P&{\mbox{ in }}B_1,\\
w=0&{\mbox{ in }}B_1^c.
\end{cases}
\end{equation}
Therefore, the uniqueness of the solution of  \eqref{KSMD-DIVv},
confronted with~\eqref{PIUUETv}, gives us that $w=\tilde u_P$, and thus~$v=u+\tilde u_P$.

This reasoning gives that the space of solutions of \eqref{DFD} is isomorphic to the space of polynomials with degree
less than or equal to $m-1$, which has dimension $N_m$, given by \eqref{soldim}.
\end{proof}

\section*{Acknowledgements}

The first and third authors are members of INdAM and AustMS.
The first author has been supported by 
the Australian Research Council DECRA DE180100957
``PDEs, free boundaries and applications''. The second author is supported by the fellowship INDAM-DP-COFUND-2015 ``INdAM Doctoral Programme in
Mathematics and/or Applications Cofunded by Marie Sklodowska-Curie Actions'', Grant 713485.
The third author has been supported by
the Australian Laureate Fellowship
FL190100081 ``Minimal surfaces, free boundaries and partial differential equations''.
Part of this work has been completed during a very pleasant visit of the second author
to the University of Western Australia, that we thank for the warm hospitality.


\begin{bibdiv}\begin{biblist}

\bib{MR4038144}{article}{
   author={Abatangelo, Nicola},
   author={Ros-Oton, Xavier},
   title={Obstacle problems for integro-differential operators: higher
   regularity of free boundaries},
   journal={Adv. Math.},
   volume={360},
   date={2020},
   pages={106931, 61},
   issn={0001-8708},
   review={\MR{4038144}},
   doi={10.1016/j.aim.2019.106931},
}

\bib{MR3967804}{article}{
   author={Abatangelo, Nicola},
   author={Valdinoci, Enrico},
   title={Getting acquainted with the fractional Laplacian},
   conference={
      title={Contemporary research in elliptic PDEs and related topics},
   },
   book={
      series={Springer INdAM Ser.},
      volume={33},
      publisher={Springer, Cham},
   },
   date={2019},
   pages={1--105},
   review={\MR{3967804}},
}

\bib{MR4135310}{article}{
   author={Audrito, Alessandro},
   author={Ros-Oton, Xavier},
   title={The Dirichlet problem for nonlocal elliptic operators with
   $C^{0,\alpha}$ exterior data},
   journal={Proc. Amer. Math. Soc.},
   volume={148},
   date={2020},
   number={10},
   pages={4455--4470},
   issn={0002-9939},
   review={\MR{4135310}},
   doi={10.1090/proc/15121},
}

\bib{BUCK}{article}{
author={Buckingham, R. A.},
date={1938},
title={The classical equation of state of gaseous helium, neon and argon},
journal={Proc. R. Soc. Lond. A},
volume={168},
pages={264--283},
doi={10.1098/rspa.1938.0173},
}

\bib{MR3770173}{article}{
   author={Cabr\'{e}, Xavier},
   author={Fall, Mouhamed Moustapha},
   author={Weth, Tobias},
   title={Near-sphere lattices with constant nonlocal mean curvature},
   journal={Math. Ann.},
   volume={370},
   date={2018},
   number={3-4},
   pages={1513--1569},
   issn={0025-5831},
   review={\MR{3770173}},
   doi={10.1007/s00208-017-1559-6},
}

\bib{CC95}{book}{
   author={Caffarelli, Luis A.},
   author={Cabr\'{e}, Xavier},
   title={Fully nonlinear elliptic equations},
   series={American Mathematical Society Colloquium Publications},
   volume={43},
   publisher={American Mathematical Society, Providence, RI},
   date={1995},
   pages={vi+104},
   isbn={0-8218-0437-5},
   review={\MR{1351007}},
   doi={10.1090/coll/043},
}

\bib{CS11}{article}{
   author={Caffarelli, Luis},
   author={Silvestre, Luis},
   title={Regularity results for nonlocal equations by approximation},
   journal={Arch. Ration. Mech. Anal.},
   volume={200},
   date={2011},
   number={1},
   pages={59--88},
   issn={0213-2230},
   review={\MR{2781586}},
   doi={10.1007/s00205-010-0336-4},
}

\bib{BOOKCA}{book}{
author = {Carbotti, Alessandro},
author = {Dipierro, Serena}
author = {Valdinoci, Enrico},
 title = {Local density of solutions to fractional equations},
 FJournal = {De Gruyter Studies in Mathematics},
 Journal = {{De Gruyter Stud. Math.}},
 ISSN = {0179-0986},
 Volume = {74},
 ISBN = {978-3-11-066069-2/hbk; 978-3-11-066435-5/ebook},
 Pages = {xi + 129},
 Year = {2019},
 Publisher = {Berlin: De Gruyter},
 MSC2010 = {35R11 35-01}
}


\bib{chierchia}{book}{
   author={Chierchia, Luigi},
   title={Lezioni di Analisi Matematica 2},
   publisher={Aracne},
   date={1997},
}

\bib{MR4303657}{article}{
   author={del Teso, F\'{e}lix},
   author={G\'{o}mez-Castro, David},
   author={V\'{a}zquez, Juan Luis},
   title={Three representations of the fractional $p$-Laplacian: Semigroup,
   extension and Balakrishnan formulas},
   journal={Fract. Calc. Appl. Anal.},
   volume={24},
   date={2021},
   number={4},
   pages={966--1002},
   issn={1311-0454},
   review={\MR{4303657}},
   doi={10.1515/fca-2021-0042},
}

\bib{MR2944369}{article}{
   author={Di Nezza, Eleonora},
   author={Palatucci, Giampiero},
   author={Valdinoci, Enrico},
   title={Hitchhiker's guide to the fractional Sobolev spaces},
   journal={Bull. Sci. Math.},
   volume={136},
   date={2012},
   number={5},
   pages={521--573},
   issn={0007-4497},
   review={\MR{2944369}},
   doi={10.1016/j.bulsci.2011.12.004},
}

\bib{MR3988080}{article}{
   author={Dipierro, Serena},
   author={Savin, Ovidiu},
   author={Valdinoci, Enrico},
   title={Definition of fractional Laplacian for functions with polynomial
   growth},
   journal={Rev. Mat. Iberoam.},
   volume={35},
   date={2019},
   number={4},
   pages={1079--1122},
   issn={0213-2230},
   review={\MR{3988080}},
   doi={10.4171/rmi/1079},
}

\bib{DSV19}{article}{
   author={Dipierro, Serena},
   author={Savin, Ovidiu},
   author={Valdinoci, Enrico},
   title={On divergent fractional Laplace equations},
   language={English, with English and French summaries},
   journal={Ann. Fac. Sci. Toulouse Math. (6)},
   volume={30},
   date={2021},
   number={2},
   pages={255--265},
   issn={0240-2963},
   review={\MR{4297378}},
   doi={10.5802/afst.1673},
}

\bib{2021arXiv210107941D}{article}{
       author = {Dipierro, Serena},
       author = {Valdinoci, Enrico},
        title = {Elliptic partial differential equations from an elementary viewpoint},
      journal = {arXiv e-prints},
date = {2021},
          eid = {arXiv:2101.07941},
        pages = {arXiv:2101.07941},
       adsurl = {https://ui.adsabs.harvard.edu/abs/2021arXiv210107941D},
}

\bib{MR3916700}{article}{
   author={Garofalo, Nicola},
   title={Fractional thoughts},
   conference={
      title={New developments in the analysis of nonlocal operators},
   },
   book={
      series={Contemp. Math.},
      volume={723},
      publisher={Amer. Math. Soc., [Providence], RI},
   },
   date={2019},
   pages={1--135},
   review={\MR{3916700}},
   doi={10.1090/conm/723/14569},
}

\bib{MR3293447}{article}{
   author={Grubb, Gerd},
   title={Local and nonlocal boundary conditions for $\mu$-transmission and
   fractional elliptic pseudodifferential operators},
   journal={Anal. PDE},
   volume={7},
   date={2014},
   number={7},
   pages={1649--1682},
   issn={2157-5045},
   review={\MR{3293447}},
   doi={10.2140/apde.2014.7.1649},
}

\bib{MR1421222}{article}{
   author={Kurokawa, Takahide},
   title={Hypersingular integrals and Riesz potential spaces},
   journal={Hiroshima Math. J.},
   volume={26},
   date={1996},
   number={3},
   pages={493--514},
   issn={0018-2079},
   review={\MR{1421222}},
}

\bib{MR3613319}{article}{
   author={Kwa\'{s}nicki, Mateusz},
   title={Ten equivalent definitions of the fractional Laplace operator},
   journal={Fract. Calc. Appl. Anal.},
   volume={20},
   date={2017},
   number={1},
   pages={7--51},
   issn={1311-0454},
   review={\MR{3613319}},
   doi={10.1515/fca-2017-0002},
}

\bib{MORS}{article}{
  title = {Diatomic molecules according to the wave mechanics. II. Vibrational levels},
  author = {Morse, Philip M.},
  journal = {Phys. Rev.},
  volume = {34},
  issue = {1},
  pages = {57--64},
  year = {1929},
  doi = {10.1103/PhysRev.34.57},
}

\bib{P18}{article}{
   author={Palatucci, Giampiero},
   title={The Dirichlet problem for the $p$-fractional Laplace equation},
   journal={Nonlinear Anal.},
   volume={177},
   date={2018},
   number={part B},
   part={part B},
   pages={699--732},
   issn={0362-546X},
   review={\MR{3886598}},
   doi={10.1016/j.na.2018.05.004},
}

\bib{R16}{article}{
   author={Ros-Oton, Xavier},
   title={Nonlocal elliptic equations in bounded domains: a survey},
   journal={Publ. Mat.},
   volume={60},
   date={2016},
   number={1},
   pages={3--26},
   issn={0214-1493},
   review={\MR{3447732}},
   doi={DOI: 10.5565/PUBLMAT\_60116\_01},
}
		
\bib{MR2707618}{book}{
   author={Silvestre, Luis Enrique},
   title={Regularity of the obstacle problem for a fractional power of the
   Laplace operator},
   note={Thesis (Ph.D.)--The University of Texas at Austin},
   publisher={ProQuest LLC, Ann Arbor, MI},
   date={2005},
   pages={95},
   isbn={978-0542-25310-2},
   review={\MR{2707618}},
}

\bib{SV14}{article}{
   author={Servadei, Raffaella},
   author={Valdinoci, Enrico},
   title={Weak and viscosity solutions of the fractional Laplace equation},
   journal={Publ. Mat.},
   volume={58},
   date={2014},
   number={1},
   pages={133--154},
   issn={0214-1493},
   review={\MR{3161511}},
   doi={10.5565/PUBLMAT\_58114\_06}
}

\bib{MR3381284}{book}{
   author={Wheeden, Richard L.},
   author={Zygmund, Antoni},
   title={Measure and integral},
   series={Pure and Applied Mathematics (Boca Raton)},
   edition={2},
   note={An introduction to real analysis},
   publisher={CRC Press, Boca Raton, FL},
   date={2015},
   pages={xvii+514},
   isbn={978-1-4987-0289-8},
   review={\MR{3381284}},
}

\bib{MR2033094}{book}{
   author={Zorich, Vladimir A.},
   title={Mathematical analysis. I},
   series={Universitext},
   note={Translated from the 2002 fourth Russian edition by Roger Cooke},
   publisher={Springer-Verlag, Berlin},
   date={2004},
   pages={xviii+574},
   isbn={3-540-40386-8},
   review={\MR{2033094}},
}
		
\end{biblist}
\end{bibdiv}
\end{document}